   \documentclass[12pt]{amsart}
\usepackage{amsmath}
\usepackage{amsthm}
\usepackage{amscd}
\usepackage{amssymb}
\usepackage{amsfonts}
\usepackage{graphics}
\usepackage{color}
\usepackage[USenglish]{babel}
\usepackage{enumitem}

\date{}

\newlength{\defbaselineskip}
\long\def\salta#1{\relax}

\numberwithin{equation}{section}

\theoremstyle{plain}
\newtheorem{theorem}{Theorem}[section]
\newtheorem{proposition}[theorem]{Proposition}
\newtheorem{lemma}[theorem]{Lemma}

\theoremstyle{definition}
\newtheorem{definition}[theorem]{Definition}

\newtheorem{remark}[theorem]{Remark}

\theoremstyle{remark}

\renewcommand{\theequation}{\thesection.\arabic{equation}}

\def\eps{\varepsilon}

\def\dys{\displaystyle}

\def\oeps{\Omega^\eps}

\def\t1p0{T^{1,p}_{0}(\Omega)}

\def\m2{M^{\frac{N(p-1)}{N-1}}(\Omega)}

\def\into{\int_{\Omega}}

\def\w-1p'{W^{-1,p'}(\Omega)}
\def\pw-1p'u{L^{p'}(0,1;W^{-1,p'}(\Omega))}

\def\dys{\displaystyle}

\def\lp'n{(L^{p'}(\Omega))^{N}}

\def\oeps{\Omega^\eps}

\numberwithin{equation}{section}

\title[A semilinear elliptic problem with a strong singularity at $u=0$]{Definition, existence, stability\\ and uniqueness of the solution\\ to a semilinear elliptic problem\\ with a strong singularity at $ u = 0 $}

\author[D. Giachetti]{Daniela Giachetti}

\address{Daniela Giachetti \newline Dipartimento di Scienze di Base e Applicate per l'Ingegneria \newline Facolt\`a di ingegneria civile e industriale\newline Sapienza Universit\`a di Roma  \newline Via Scarpa 16, 00161 Roma, Italy}
\email{{\tt daniela.giachetti@sbai.uniroma1.it}}

\author[P.J. Mart\'inez-Aparicio]{Pedro J. Mart\'inez-Aparicio}

\address{Pedro J. Mart\'inez-Aparicio \newline Departamento de Matem{\'a}tica Aplicada y Estad\'istica \newline
Universidad Polit{\'e}cnica de Cartagena \newline Paseo Alfonso XIII 52, 30202 Cartagena (Murcia), Spain}
\email{{\tt pedroj.martinez@upct.es}}

\author[F. Murat]{Fran\c cois Murat}
\address{Fran\c cois Murat \newline Laboratoire Jacques-Louis Lions et CNRS \newline Universit\'e Pierre et Marie Curie \newline Bo\^ite Courrier 187, 75252 Paris Cedex 05, France}
\email{{\tt murat@ann.jussieu.fr}}


\date{}

\begin{document}
\maketitle








\centerline{version April 09, 2017}

\centerline{Accepted for publication in Ann. Sc. Norm. Sup. Pisa}

\begin{abstract}
\label{abs}
In this paper we consider a semilinear elliptic equation with a strong singularity at $u=0$, namely
\begin{equation*}
\begin{cases}
\dys u\geq 0 & \mbox{in }  \Omega,\\
\displaystyle - div \,A(x) D u  = F(x,u)& \mbox{in} \; \Omega,\\
u = 0 & \mbox{on} \; \partial \Omega,\\
\end{cases} 
\end{equation*}
with $F(x,s)$ a Carath\'eodory function such that 
$$
0\leq F(x,s)\leq \frac{h(x)}{\Gamma(s)}\,\,\mbox{ a.e. } x\in\Omega,\, \forall s>0,
$$
with $h$ in some $L^r(\Omega)$ and $\Gamma$ a $C^1([0,+\infty[)$ increasing function such that $\Gamma(0)=0$ and $\Gamma'(s)>0$ for every $s>0$.

 We introduce a notion of solution to this problem in the spirit of the solutions defined by transposition. This definition  allows us to prove the existence and the stability of this solution, as well as its uniqueness when $F(x,s)$ is nonincreasing in $s$.

\end{abstract}

\noindent {\bf Contents}\label{contents}

\medskip

\noindent Abstract \dotfill \pageref{abs}

\noindent 1 Introduction \dotfill  \pageref{Intro}

\medskip

\noindent 2 Assumptions  \dotfill \  \pageref{assumptions}

2.1 Notation  \dotfill \pageref{notation}

\medskip

\noindent 3 Definition of a solution to problem \eqref{eqprima}   \dotfill \pageref{secdefi}

3.1 The space $\mathcal V(\Omega)$ of test functions   \dotfill 
\pageref{secdefiV}

3.2 Definition of a solution to problem \eqref{eqprima}
\dotfill  \pageref{sub32}

\medskip

\noindent 4 Statements of the existence, stability and uniqueness results \dotfill  \pageref{stat}

\medskip

\noindent 5 A priori estimates \dotfill \pageref{estimates}

5.1 A priori estimate of $G_k(u)$ in $H_0^1(\Omega)$ \dotfill  \pageref{sub51}

5.2 A priori estimate of $\varphi DT_k(u)$ in $(L^2(\Omega))^N$ for $\varphi \in H_0^1(\Omega)\cap L^\infty(\Omega)$  \dotfill  \pageref{sub52}

5.3 Control of the integral $\dys\int_{\{u\leq \delta\}}F(x,u)v$  \dotfill  \pageref{sub53}

5.4 A priori estimate of $\beta(u)$ in $H_0^1(\Omega)$ \dotfill  \pageref{sub54}

\medskip

\noindent 6 Proofs of the Stability Theorem~\ref{est} and of the Existence Theorem~ \ref{EUS} \dotfill  \pageref{proofexistence}

6.1 Proof of the Stability Theorem~\ref{est} \dotfill  \pageref{proofstability}

6.2 Proof of  the Existence Theorem~\ref{EUS} \dotfill  \pageref{profiste}

\medskip

\noindent 7  Comparison Principle and proof of the Uniqueness Theorem~\ref{uniqueness} $\;$ \dotfill  \pageref{comparison}

\medskip

\noindent Appendix~\ref{appendixa} An useful lemma \dotfill  \pageref{appendixa}

\medskip
 
\noindent Acknowledgments   \dotfill  \pageref{acknowledgments}

\noindent References   \dotfill  \pageref{references}

\vskip1cm

\section{Introduction}
\label{Intro}

\noindent {\bf  Position of the problem}
 
In the present paper we deal with a semilinear problem with a strong singularity at $u=0$, which consists in finding a function $u$ which satisfies 
\begin{equation}
\label{eqprimai}
\begin{cases}
\dys u\geq 0 & \mbox{in }  \Omega,\\
\displaystyle - div\, A(x) D u  = F(x,u) & \mbox{in} \; \Omega,\\
u = 0 & \mbox{on} \; \partial \Omega,\\
\end{cases} 
\end{equation}
where $\Omega$ is a bounded open set of $\mathbb{R}^N$,  $N\geq 1$, where $A$ is a coercive matrix with coefficients in $L^\infty(\Omega)$, and where 
$$
F:(x,s)\in\Omega\times [0,+\infty[\to F(x,s)\in[0,+\infty]
$$ 
is a Carath\'eodory function which satisfies 
\begin{equation}
\label{condi}
\displaystyle 0 \leq F(x,s)\leq \frac{h(x)}{\Gamma(s)}\, \mbox{ a.e. } x\in\Omega,\, \forall s>0,
\end{equation}
with
\begin{align}
\begin{cases}
\label{gami}
h\geq 0,\\
h\in L^r(\Omega),\, r=\frac{2N}{N+2}\,\,\,\text {if}\,\, N\geq 3,\,\, r>1 \,\,\text{if} \,\, N=2,\,\,r=1\,\, \text {if} \,\,N=1,\\
\Gamma:s\in[0,+\infty[ \longrightarrow \Gamma(s)\in[0,+\infty[ \mbox{ is a } C^1([0,+\infty[) \mbox{ function}\\
\mbox{such that } \Gamma(0)=0 \mbox{ and } \Gamma'(s)>0 \,\, \forall s>0.
\end{cases}
\end{align}

\vspace{0.5cm}
\noindent {\bf The function $F$}


A model for the function $F(x,s)$ (see a more general model in \eqref{21bis} and Remark~\ref{primerarem} {\it viii}) below)  is given by 
\begin{align}
\begin{cases}
\label{modeli}
\dys F(x,s)= f(x) \frac{\left(a+ \sin (\frac 1s) \right)}{\exp (- \frac {1}{s})} + g(x) \dys\frac{\left(b+ \sin (\frac 1s) \right)}{s^\gamma} + l(x)\\
 \dys \mbox{a.e. } x\in\Omega,\, \forall s>0,
\end{cases}
\end{align}
where $\gamma>0,$ $a>1,$ $b>1$ and where the functions $f$, $g$ and $l$ are nonnegative. In this model 
$\dys\frac{\left(a+ \sin (\frac 1s) \right)}{\exp (- \frac {1}{s})}$, 
as well as 
$\dys\frac{\left(b+ \sin (\frac 1s) \right)}{s^\gamma}$,
are just examples of functions which can be replaced by any singularity $\dys\frac{1}{\Gamma(s)}$ with $\Gamma$ satisfying \eqref{gami}.  Note that the behaviour of $F(x,s)$ for $s=0$ can be very different according to the point $x$.

Note also that the function $F(x,s)$ is defined only for $s\geq0$, and that in view of \eqref{condi} and \eqref{gami}, $F(x,s)$ is finite for almost everywhere $x\in\Omega$ and for every $s>0$, but that $F(x,s)$ can exhibit a singularity when $s>0$ tends to $0$ and when $h(x)>0$. Let us emphasize the fact that the Carath\'eodory character of the function $F(x,s)$, which can take infinite values when $s=0$, means in particular that for almost every $x\in\Omega$, the function $s\to F(x,s)$ is continuous not only for every $s>0$ but also for $s=0$.

Note finally that we do not require $F(x,s)$ to be nonincreasing in $s$, except when we deal with uniqueness and comparison results  in the Uniqueness Theorem~\ref{uniqueness} and in Section~\ref{comparison} below. 

\vspace{0.5cm}
\noindent {\bf The case of a mild singularity}

In the present paper we consider the case of strong singularities, namely the case where  $F(x,s)$ can have any (wild) behavior as $s$ tends to zero, while in our previous paper \cite{GMM1} we restricted ourselves to the case of mild singularities, namely the case where 
\begin{align}
\begin{cases}
\label{15bis}
\dys0\leq F(x,s)\leq h(x)\left(\frac{1}{s^\gamma}+1\right)\, \mbox{ a.e. } x\in\Omega,\, \forall s>0\\\mbox{with } 0<\gamma\leq 1.
\end{cases}
\end{align}

\addtocounter{equation}{1}
When the function $F$ satisfies \eqref{15bis}, our definition of the solution to problem \eqref{eqprimai}  is relatively classical, because it consists in looking for a function $u$ such that
\begin{align}
\tag{{\theequation}$_{\,\,\mbox{\tiny{mild}}}$}
\label{gamin}
\begin{cases}
u\in H_0^1(\Omega),\\
u\geq 0 \mbox{ a.e. in } \Omega,\\
\dys  \!\into F(x,u)\varphi<+\infty\,\, \forall \varphi\in H_0^1(\Omega),\, \varphi\geq 0,\vspace{0.1cm}
\\
\dys \,\into A Du D\varphi=\into F(x,u)\varphi \,\, \forall \varphi\in H_0^1(\Omega),
\end{cases}
\end{align}
(see \cite[Section~3]{GMM1}). 

\vspace{0.5cm}
\noindent {\bf Definition of the solution for a strong singularity}

\noindent In contrast, the definition of the solution to problem \eqref{eqprimai} that we use in the present paper is less classical. 
This definition is in our opinion one of the main originalities of the present paper. 

We refer the reader to Section~\ref{secdefi} below where this definition is given in details, but here we emphasize some of its main features.  

When the function $F$ satisfies \eqref{condi} and \eqref{gami}, we will say (see Definition~\ref{sol} below) that $u$ is a solution to problem \eqref{eqprimai} if 
\addtocounter{equation}{1}
\begin{equation}
\tag{{\theequation}$_{\,\,\mbox{\tiny{strong}}}$}
\label{sol1i}
\begin{cases}
i)\, u\in L^2(\Omega)\cap H^1_{\mbox{{\tiny loc}}}(\Omega),\\
ii)\, u(x)\geq 0 \,\,\mbox{ a.e. } x\in\Omega,\\
iii)\, G_k(u)\in H_0^1(\Omega)\,\, \,\, \forall k>0,\\
iv)\, \varphi T_k(u)\in H^1_0(\Omega)\,\, \,\,\forall k>0,\,\, \forall \varphi \in H^1_0(\Omega)\cap L^\infty (\Omega),
\end{cases}
\end{equation}
where $G_k(s)=(s-k)^+$ and $T_k(s)=\inf(s,k)=s-G_k(s)$ for $s>0$, and if

\addtocounter{equation}{1}
\begin{align}
\tag{{\theequation}$_{\,\,\mbox{\tiny{strong}}}$}
\label{sol2i}
\begin{cases}
\displaystyle\forall v \in H_0^1(\Omega)\cap L^\infty(\Omega),\,v\geq 0,\\ \dys\mbox{with } -div \, {}^t\!A(x)Dv=\sum_{i \in I} \hat{ \varphi}_i (-div\, \hat{ g}_i)+\hat{ f} \mbox{ in } \mathcal{D}'(\Omega),\\
\mbox{where } I \mbox{ is finite, } \hat{\varphi_i}\in H_0^1(\Omega)\cap L^\infty(\Omega), \hat{g_i}\in (L^2(\Omega))^N, \hat{f_i}\in L^1(\Omega),\\
\mbox{one has}\\\vspace{0.1cm}
i) \,\dys  \into F(x,u) v<+\infty,\\
ii) \dys \into\, {}^t\!A(x)Dv DG_k(u)+  \displaystyle\sum_{i\in I} \into \hat{g_i} D(\hat{\varphi_i} T_k(u))+\into \hat{f} T_k(u)=\\=\langle -div\, {}^t\!A(x)Dv, G_k(u) \rangle_{H^{-1}(\Omega),H_0^1(\Omega)}+\langle\langle -div\, {}^t\!A(x)Dv, T_k(u)\rangle\rangle_{\Omega}=\\
\displaystyle =\into F(x,u) v\,\,\,\forall k>0,
\end{cases}
\end{align}
where $\langle\langle \mbox{ }, \rangle\rangle_\Omega$ is a notation introduced in \eqref{dc} below. 

 Note that in this definition every term of (\ref{sol2i} {\it ii}) makes sense in view of (\ref{sol1i} {\it iii}) and (\ref{sol1i} {\it iv}).

\vspace{0.5cm}
This definition \eqref{sol1i}, \eqref{sol2i} strongly differs from the definition \eqref{gamin}. Indeed for example the assertion $u\in H_0^1(\Omega)$ of \eqref{gamin} is replaced in \eqref{sol1i} by  $u\in H^1_{\mbox{\tiny loc}}(\Omega)$, $G_k(u)\in H_0^1(\Omega)$ for every $k>0$ and $\varphi T_k(u)\in H_0^1(\Omega)$ for every $k>0$ and for every $\varphi \in H^1_0(\Omega)\cap L^\infty (\Omega)$. Here the assertion $G_k(u)\in H_0^1(\Omega)$ for every $k>0$ in particular expresses the boundary condition $u=0$ on $\partial\Omega$ (see Remark~\ref{rem25} {\it ii})) below.

Actually, the solution $u$ does not in general belong to $H_0^1(\Omega)$, or in other terms $u$ ``does not belong to $H^1(\Omega)$ up to the boundary", when the function $F$ exhibits a strong singularity at $s=0$. It is indeed proved by A.C. Lazer and P.J. McKenna in \cite[Theorem~2]{LM} that when $f\in C^\alpha(\overline{\Omega})$, $\alpha>0$, with 
$
f(x)\geq f_0>0 \mbox{ in } \Omega,
$
in a domain $\Omega$ with $C^{2+\alpha}$ boundary, the solution $u$ to the equation

\begin{equation*}
\begin{cases}
\dys -\Delta u =\frac{f(x)}{u^\gamma}  & \mbox{in }  \Omega,\\
u = 0 & \mbox{on} \; \partial \Omega,\\
\end{cases} 
\end{equation*}
does not belong to $H^1(\Omega)$ when $\gamma>3$, even if $u$ belongs to $C^{2+\alpha}(\Omega)\cap C^0(\overline{\Omega})$ in this case.

The ``space" defined by \eqref{sol1i} in which we look for a solution to \eqref{eqprimai} in the case of a strong singularity is therefore fairly different from the space $H_0^1(\Omega)$ in which we look for a solution to \eqref{eqprimai} in the case of a mild singularity.

\vspace{0.5cm}

Another important difference between the two definitions appears as far as the partial differential equation in \eqref{eqprimai} is concerned. Indeed, if the formulation in the last line of \eqref{gamin} is classical, with the use of test functions in $H_0^1(\Omega)$, the equation in \eqref{eqprimai} as formulated in \eqref{sol2i} involves test functions which belong to the vectorial space described in the three first lines of \eqref{sol2i}, that we will denote by $\mathcal{V}(\Omega)$ in the rest of the present paper.  Actually this space $\mathcal{V}(\Omega)$ consists in functions $v$ such that $ -div\, {}^t\!A(x)Dv$ can be put in the usual  duality between $H^{-1}(\Omega)$ and $H^1_0(\Omega)$ with $G_k(u)$, and in a ``formal duality" (through the notation $\langle\langle\mbox{ , } \rangle\rangle_{\Omega}$ introduced in \eqref{dc} below) with $T_k(u)$: indeed, when $u$ satisfies  \eqref{sol1i}, one writes $u$ as the sum $u=G_k(u)+T_k(u)$, where  $G_k(u)\in H^1_0(\Omega)$ and where $\varphi T_k(u)\in H_0^1(\Omega)$ for every $k>0$ and for every  $\varphi \in H_0^1(\Omega)\cap L^\infty(\Omega)$; on the other hand $ -div\, {}^t\!A(x)Dv$, which of course belongs to $H^{-1}(\Omega)$ when $v\in H_0^1(\Omega)$, will be assumed to be the sum of a function $\hat{f} \in L^1(\Omega)$ and of a finite sum of functions $\hat{\varphi_i}(-div\, \hat{g_i})$, with $\hat{\varphi_i}\in H_0^1(\Omega)\cap L^\infty(\Omega)$  and $ -div\, \hat{g_i}\in H^{-1}(\Omega) $, a fact which allows one to correctly define the ``formal duality"   $\langle\langle-div\, {}^t\!A(x)Dv,T_k(u) \rangle\rangle_{\Omega}$ (see \eqref{dc} in Definition~\ref{def32} below). 

This is a definition of the solution by transposition in the spirit of those introduced by J.-L. Lions and E. Magenes and by G. Stampacchia. Here again things are fairly different with respect to the case of a mild singularity. 

\vspace{0.5cm}
\noindent {\bf Main results of the present paper}

In the framework of the definition \eqref{sol1i}, \eqref{sol2i} (i.e. of Definition~\ref{sol} below) we are able to prove that there exists at least a solution to problem \eqref{eqprimai} (see the Existence Theorem~\ref{EUS} below).  We also prove that this solution is stable with respect to variations of the right-hand side (see the Stability Theorem~\ref{est} below). Finally, if further to \eqref{condi} and \eqref{gami} we assume that the function $F(x,s)$ is  nonincreasing with respect to $s$, then this solution is unique (see the Uniqueness Theorem~\ref{uniqueness} below). 

In brief Definition~\ref{sol} below provides a framework where problem \eqref{eqprimai} is well posed in the sense of Hadamard.

\vspace{0.5cm}

Moreover every solution $u$ to problem \eqref{eqprimai} defined in this sense satisfies the following a priori estimates: 

\begin{itemize}[leftmargin=*]
\item an a priori estimate of $G_k(u)$ in $H_0^1(\Omega)$ for every $k>0$ (see Proposition~\ref{prop1} below) which is formally obtained by using $G_k(u)$ as test function in \eqref{eqprimai}; this implies an a priori estimate of $u$ in $L^2(\Omega)$ (see Remark~\ref{rem52u} below);

\item an a priori estimate of $\varphi DT_k(u)$ in $(L^2(\Omega))^N$ for every $\varphi \in H^1_0(\Omega)\cap L^\infty (\Omega)$ (see Proposition~\ref{prop2} below) which is formally obtained by using $\varphi^2 T_k(u)$ and $\varphi^2$ as test functions in \eqref{eqprimai}; this implies an a priori estimate of $u$ in $H^1_{\mbox{\tiny loc}}(\Omega)$ (see Remark~\ref{55bis} below) and an a  priori estimate of $\varphi T_k(u)$ in $H_0^1(\Omega)$ for every $k>0$ and for every $\varphi \in H^1_0(\Omega)\cap L^\infty (\Omega)$ (see 
Remark~\ref{54bis} below);

\item  an a priori estimate of the integral  $\dys \int_{\{u\leq \delta\}} F(x,u)v$ for every $v\in\mathcal{V}(\Omega)$, $v\geq 0$, and every  $\delta>0$, which depends only on $\delta$ and on $v$ (and also on $u$ through the products $\hat{\varphi}_iDu$) by a constant which tends to zero when $\delta$ tends to zero (see  Proposition~\ref{prop3} and Remark~\ref{510bis} below);

\item  an a priori estimate of $ \beta(u)$ in $H^1_0(\Omega)$ (where the function $\beta$ is defined from the function $\Gamma$ by $\dys\beta(s)=\int_0^s \sqrt{\Gamma'(t)}dt$)   (see Proposition~\ref{lem2} below) which is formally obtained by using  $\Gamma(u)$ as test function in \eqref{eqprimai}.
\end{itemize}

\noindent The (mathematically correct) proofs of all the above a priori estimates are based on Definition \eqref{sol1i}, \eqref{sol2i} and on the use of convenient test functions $v$ in \eqref{sol2i}.


\vspace{0.5cm}

\noindent {\bf Literature} 

 There is a wide literature dealing with problem \eqref{eqprimai}. We will not pretend to give here a complete list of references and we will concentrate on some papers which seem to us to be  the most significant, also refering the interested reader to the references quoted there.

The problem \eqref{eqprimai} was initially proposed in 1960 in the pioneering work \cite{FM} of W.~Fulks and J.S.~Maybe as a model for several physical situations. The problem was then studied by many authors, among which we will quote the works of  C.A.~Stuart \cite{St},  M.G.~Crandall, P.H.~Rabinowitz and L.~Tartar \cite{BCR},  J.I.~Diaz, J.M.~Morel and L.~Oswald \cite{DiMoOs},  M.M.~Coclite and G.~Palmieri \cite{CP}, A.C.~Lazer and P.J.~McKenna \cite{LM}, Y.S.~Choi and P.J.~McKenna \cite{CM},  J.~Davila and M.~Montenegro \cite{DaMo},  C.O.~Alves, F.J.S.A.~Correa and J.V.A.~Goncalves \cite{ACG}, F.~Cirstea, M.~Ghergu and V.~Radulescu \cite{CGR},   G.M.~Coclite and M.M.~Coclite \cite{CoCo}, D.~Arcoya and L.~Moreno-M\'erida \cite{AMo},  and  F.~Oliva and F.~Petitta \cite{OlPe}.

 In most of the above quoted papers the authors look for a strong solution and use sub- and super-solutions.  In particular in \cite{LM}  A.C.~Lazer and P.J.~McKenna work in $C^{2,\alpha}(\Omega)$ and $W^{2,q}(\Omega)$  and use methods of sub- and super-solutions, proving  that when $\dys F(x,s)=\frac{f(x)}{s^\gamma}$ with $\gamma>0$, $f\in C^\alpha(\overline{\Omega})$ and $f(x)\geq f_0>0$ in a $C^{2,\alpha}$ domain $\Omega$, then one has 
$
c_1\phi_1(x)\leq u(x)^{\frac{\gamma+1}{2}}\leq c_2\phi_1(x)
$
for two constants $0<c_1<c_2$, where $\phi_1$ is the first (positive) eigenfunction of $-div\, A(x)D$ in $H_0^1(\Omega)$; 
M.G.~Crandall, P.H.~Rabinowitz and L.~Tartar \cite{BCR} study the behaviours of $u(x)$ and $|Du(x)|$ at the boundary. Let us finally note that C.~Stuart \cite{S},  as well as  M.G.~Crandall, P.H.~Rabinowitz and L.~Tartar \cite{BCR}, do not assume that $F(x,s)$ is nonincreasing in $s$.

More recently L.~Boccardo and L.~Orsina studied  in \cite{BO} the problem in the framework of weak solutions in the sense of distributions. In that paper, the authors address the problem \eqref{eqprimai} with $\dys F(x,s)=\frac{f(x)}{s^\gamma}$, where $\gamma>0$ and where $f\geq 0$ belongs to $L^1(\Omega)$, or to other Lebesgue's spaces, or to the space of Radon's measures, and they prove existence and regularity as well as non existence results. In this work the strong maximum principle and the nonincreasing character of the function $\dys F(x,s)=\frac{f(x)}{s^\gamma}$ with respect to $s$ play prominent roles. The solution $u$ to problem \eqref{eqprimai} is indeed required to satisfy $u(x)\geq c(\omega,u)>0$ on every open set $\omega$ with $\overline{\omega}\subset \Omega$, and the framework in which the solution is searched for is the Sobolev's space  $H_0^1(\Omega)$ or $H^1_{\mbox{\tiny loc}}(\Omega)$. Then L.~Boccardo and J.~Casado-D\'iaz  proved in \cite{BC} the uniqueness of the solutions obtained by approximation and the stability of the solution with respect to the $G$-convergence of a sequence of matrices $A^\eps(x)$ which are equicoercive and equibounded.

For a variational approach to the problem and extensions to the case of a nonlinear principal part, see A.~Canino and M.~Degiovanni  \cite{CaDe}, L.M.~De~Cave \cite{Cave}, and A.~Canino, B.~Sciunzi and A.~Trombetta \cite{CaScTro}. 

In the very recent preprint \cite{OlPe}, F.~Oliva and F.~Petitta consider the case where $F(x,s)=f(x)g(s)$ with $f$ a nonnegative measure and $g$ a continuous function which is singular at $u=0$; they prove in particular the existence of a solution $u\in L^1(\Omega)\cap W{\mbox{{\tiny loc}}}^{1,1}(\Omega)$ which satisfies the equation in the sense of distributions, and, when $g$ is nonincreasing, the uniqueness of such a solution by using convenient solutions to the adjoint problem. \vspace{0.5cm}

\noindent {\bf Contributions of the present paper}
 
 In the present paper (as in \cite{GMM1}, where mild solutions are considered), we obtain in particular an a priori estimate of the singular term in the region where the solution is close to zero. 
This estimate, which is new, is an essential tool in our proof. 
Moreover, in the case of strong singularities, the main difficulty, with respect to the case of mild singularities, is that only local estimates in the energy space are available for the solutions (see \cite{LM}). In order to solve this difficulty, we prove that the solution satisfies $\varphi Du\in (L^2(\Omega))^N$ for every $\varphi\in H_0^1(\Omega)\cap L^\infty(\Omega)$, which is a property of the solution which was not known before. We also introduce a convenient class of test functions, namely the class $\mathcal V(\Omega)$  defined in Subsection 3.1 below. The class $\mathcal V(\Omega)$  and the method  of proof that we use seem to be rather flexible and can be adapted to other situations where other techniques 
would fail. This is the case for the existence of solutions when the equation involves a zeroth-order term which prevents the use of the strong maximum principle in the equation (see \cite{GMM2bis}), as well as for the  homogenization of singular problems in perforated domains  $\oeps$ obtained from $\Omega$ by removing many small holes, when the Dirichlet boundary condition on $\partial\oeps$ leads to the appearance of a ``strange term" $\mu u$ in $\Omega$ (see \cite{GMM3}).
 
One of the strong points of the present paper is Definition~\ref{sol}. This definition makes problem \eqref{eqprimai} well posed in the sense of Hadamard when the function $F(x,s)$ is nonincreasing in $s$, and allows us to perform in a mathematically correct way all the formal computations that we want to make on problem \eqref{eqprimai}, even when the assumption that the function $F(x,s)$ is nonincreasing in $s$ is not done.
Two other strong points are the fact that we do not assume that the function $F(x,s)$ is nonincreasing with respect to $s$, except as far as uniqueness is concerned, and that we do not use the strong maximum principle. 

A weak point of the present paper is however the fact that Definition~\ref{sol} is a definition by transposition in the spirit of J.-L.~Lions and E.~Magenes and of G.~Stampacchia, a feature which makes difficult (if not impossible, except maybe when $p=2$) to extend it to the case of general nonlinear monotone operators. 
  \smallskip

\section{Assumptions}
\label{assumptions}

As said in the Introduction we study in this paper solutions to the following singular semilinear problem
\begin{equation}
\label{eqprima}
\begin{cases}
u\geq 0 & \mbox{in } \Omega,\\
\displaystyle - div\, A(x) D u  = F(x,u) & \mbox{in} \; \Omega,\\
u = 0 & \mbox{on} \; \partial \Omega,\\
\end{cases} 
\end{equation}
where a model for the function $F(x,s)$ is given by \eqref{modeli}; another more general model is the following one
\begin{align}
\label{21bis}
\begin{cases}
\dys F(x,s)= f(x) \frac{\left(a+ \sin (S(s)) \right)}{\exp (- S(s))} + g(x) \dys\frac{\left(b+ \sin (\frac 1s) \right)}{s^\gamma} + l(x)\\
 \dys \mbox{a.e. } x\in\Omega,\, \forall s>0,
\end{cases}
\end{align}
where $\gamma>0$, $a>1$, $b>1$, 
where the function $S$ satisfies 
\begin{equation}
\label{eq.2.2bis} 
S \in C^1(]0,+\infty[),     \quad   S'(s) < 0 \,\, \forall s>0,   \quad   S(s) \to + \infty  \,\, \mbox{as }  s \to 0,
\end{equation}
and where the functions $f$, $g$ and $l$ are nonnegative and belong to $L^r(\Omega)$ with $r$ defined in (\ref{eq0.1} {\it i})
(see  Remark~\ref{primerarem} {\it viii}) below).

\bigskip
In this section, we give the precise assumptions that we make on the data of problem \eqref{eqprima}.
\bigskip

We assume that $\Omega$ is an open bounded set of $\mathbb{R}^N,\, N\geq 1$ (no regularity is assumed on the boundary $\partial\Omega$ of $\Omega$), that the matrix $A$ is bounded and coercive, i.e. satisfies
\begin{equation}\label{eq0.0}
A(x)\in (L^\infty(\Omega))^{N\times N},\,\,
\exists \alpha>0,\, \,A(x)\geq \alpha I \,\,\,\,\,{\rm a.e.}\,\,  x\in\Omega,
\end{equation}
and that the function $F$ satisfies
\begin{equation}
\label{car} 
\begin{cases}
F: (x,s)\in\Omega\times [0, +\infty[ \rightarrow F(x,s)\in [0, +\infty] \,\, \text {is a Carath\'eodory function},\\ \mbox{i.e. } F \mbox{ satisfies}\\
i)\,\forall s\in [0,+\infty[,\, x\in\Omega\to F(x,s)\in [0,+\infty] \mbox{ is measurable},\\
ii)\, \mbox{for a.e. } x\in\Omega, \, s\in [0,+\infty[\rightarrow F(x,s)\in [0,+\infty] \mbox{ is continuous},
\end{cases}
\end{equation}
\begin{equation}
\label{eq0.1}
\begin{cases}
i) \,\exists h,  h(x)\geq 0 \, \, {\rm a.e.}\,\, x \in \Omega, h \in L^r(\Omega),\\
\dys\mbox{with }\, r=\frac{2N}{N+2}\,\,\,\text {if}\,\, N\geq 3,\,\, r>1 \,\,\text{if} \,\, N=2,\,\,r=1\,\, \text {if} \,\,N=1,\\
ii)\,\exists\Gamma: s\in[0,+\infty[ \rightarrow \Gamma(s)\in [0,+\infty[,  \,\,\Gamma\in C^1([0,+\infty[),\\ 
\mbox{such that } \Gamma(0)=0 \mbox{ and } \Gamma'(s)>0\,\, \forall s>0,\\
iii) \, \displaystyle 0 \leq F(x,s)\leq  \frac{h (x)}{\Gamma(s)}\,\,  \mbox{a.e. } x \in \Omega, \forall s>0.
\end{cases}
\end{equation}




Moreover, when we will prove comparison and uniqueness results (Proposition~\ref{prop0} and Theorem~\ref{uniqueness}),
we will assume that $F(x,s)$ is  nonincreasing in $s$, i.e. that
\begin{equation}\label{eq0.2}
\dys F(x,s)\leq F(x,t)\,\,\,{\rm a.e.}\,\,\, x \in \Omega,\,\,\forall s,\forall t,\,0\leq t\leq s. 
\end{equation}



\begin{remark}
\label{primerarem}
\begin{itemize}[leftmargin=*]
\mbox{ }
\item {\it i}) If a function $\Gamma=\Gamma(s)$ satisfies (\ref{eq0.1} {\it ii}), then $\Gamma$ is (strictly) increasing and satisfies $\Gamma(s)>0$ for every $s>0$; note that the function $\Gamma$ can be either bounded or unbounded.

Observe also that if a function $F=F(x,s)$ satisfies \eqref{eq0.1} for $h(x)$ and $\Gamma(s)$, and if $\overline{\Gamma}(s)$ is a function which satisfies (\ref{eq0.1} {\it ii}) and $\overline{\Gamma}(s)\leq \Gamma(s)$, then $F(x,s)$ satisfies \eqref{eq0.1} for $h(x)$ and $\overline{\Gamma}(s)$.

\smallskip
\item {\it ii}) The function $F=F(x,s)$ is a nonnegative Carath\'eodory function with values in $[0,+\infty]$ and not only in $[0,+\infty[$. But, in view of condition
(\ref{eq0.1} {\it iii}), for almost every $x\in\Omega$, the function $F(x,s)$ can take the value $+\infty$ only when $s=0$ (or, in other terms, $F(x,s)$ is finite for almost every $x\in\Omega$ when $s>0$).
\smallskip
\item {\it iii}) Note that (\ref{eq0.1} {\it ii})  and (\ref{eq0.1} {\it iii}) do not impose any restriction on the growth of $F(x,s)$ as $s$ tends to zero. Moreover it can be proved (see \cite[Section~3, Proposition~1]{GMM2}) that \eqref{eq0.1}  is equivalent to the following assumption
\begin{equation}
\label{12ter}
\begin{cases}
\forall k>0,\,\exists h_k,\, h_k(x)\geq 0\mbox{ a.e. } x\in\Omega,\\
h_k\in L^r(\Omega) \mbox{ with $r$ as in  (\ref{eq0.1} {\it i})  such that}\\
0\leq F(x,s)\leq h_k(x) \mbox{ a.e. } x\in\Omega,\, \forall s\geq k.
\end{cases}
\end{equation}   
\smallskip
\item {\it iv}) Note that no growth condition is imposed from below on $F(x,s)$ as $s$ tends to zero. Indeed it can be proved (see \cite[Section~3, Remark~2]{GMM2}) that for every given functions $G_1(s)$ and $G_2(s)$ which satisfy \eqref{eq0.1}, with $h_1(x)=h_2(x)=1$, $G_1(s)\leq G_2(s)$ and $\dys\frac{G_2(s)}{G_1(s)}$ which tends to infinity as $s$ tends to zero, there exist a function $F(s)$ and two sequences $s_n^1$ and $s_n^2$ which tend to zero such that $F(s_n^1)=G_1(s_n^1)$ and $F(s_n^2)=G_2(s_n^2)$.

Of course the growth of $F(x,s)$ as $s$ tends to zero can (strongly) depend on the point $x\in\Omega$.

\smallskip
\item {\it v}) Note that the growth condition (\ref{eq0.1} {\it iii}) is stated for every $s>0$, while in \eqref{car}  $F$ is supposed to be a Carath\'eodory function defined for $s$ in $[0,+ \infty[$ and not only in $]0,+\infty[$. Indeed an indeterminacy $\dys\frac{0}{0}$ appears in $\dys \frac{h(x)}{\Gamma(s)}$ when $h(x)=0$ and $s=0$, while the growth and Carath\'eodory assumptions imply that 
$$
F(x,s)=0 \quad \forall s\geq 0 \quad \mbox{a.e. on } \{x\in\Omega:h(x)=0\}.
$$

In contrast, when $h$ is assumed to satisfy  $h(x)>0$ for almost every $x\in\Omega$, it is equivalent to have  (\ref{eq0.1} {\it iii}) for every $s>0$ or for every $s\geq 0$.
\smallskip

\item {\it vi})  Let us observe that the functions $F(x,s)$ given in examples \eqref{modeli} and \eqref{21bis} satisfy assumption \eqref{eq0.1}; indeed for these examples one has 
$$
0\leq F(x,s)\leq \overline{h}(x)\left( \frac{1}{\overline{\Gamma}(s)}+1\right)
$$
for some $\overline{h}(x)$ and $\overline{\Gamma}(s)$ which satisfy (\ref{eq0.1} {\it i}) and (\ref{eq0.1} {\it ii}); taking $\dys \Gamma(s)=\frac{\overline{\Gamma}(s)}{1+\overline{\Gamma}(s)}$ it is clear that $\Gamma(s)$ satisfies (\ref{eq0.1} {\it ii}) and that $F(x,s)$ satisfies \eqref{eq0.1}.

\smallskip

\item {\it vii}) The $C^1$ regularity of the function $\Gamma$ which is assumed in (\ref{eq0.1} {\it ii}) is used to define the function $\beta$ which appears in Proposition~\ref{lem2} (see \eqref{62bis}). The latest (regularity) result of $\beta(u)$ is in turn strongly used in the proofs of the Comparison Principle of Proposition~\ref{prop0} and of the Uniqueness Theorem~\ref{uniqueness}.

This $C^1$ regularity could appear as a strong restriction on the function $\Gamma$, but it can actually be proved (see \cite[Section~3]{GMM2}) that, given a function $\overline{\Gamma}$ such that for some $\delta>0$ and some $M>0$
\begin{align*}
\begin{cases}
\dys \overline{\Gamma}\in C^0([0,+\infty[),\,\, \overline{\Gamma}(0)=0,\,\, \overline{\Gamma}(s)>0 \,\,\,\, \forall s>0,\\
\dys \overline{\Gamma}(s)\geq \delta\,\,\,\, \forall s\geq M,
\end{cases}
\end{align*}
one can construct a function $\Gamma$ which satisfies (\ref{eq0.1} {\it ii}) as well as $\Gamma(s)\leq \overline{\Gamma}(s)$ for every $s>0$ and then use Remark~\ref{primerarem} {\it i}) above.

\smallskip 

\item {\it viii}) As far as example (\ref{21bis}) is concerned, 
note that when $a>1$ and $b>1$ and when the function $S$ satisfies (\ref{eq.2.2bis}), the functions 
$\dys\frac{\left(a+ \sin (S(s)) \right)}{\exp (- S(s))}$
and $\dys\frac{\left(b+ \sin (\frac 1s) \right)}{s^\gamma}$
are positive and continuous in $[0,+\infty[$
(this is no more the case when $a=1$ or $b=1$),
and therefore satisfy (\ref{car}).
Note also that these functions are easily shown to be oscillatory (in the sense that they are not nondecreasing) when $1 < a < \sqrt{2}$ and $b > 1$.
Note finally that every function $\Gamma$ which satisfies (\ref{eq0.1} {\it ii}) can be written as 
$\Gamma (s) = \exp (- S(s))$
for a function $S$ which satisfies (\ref{eq.2.2bis}).

\end{itemize}
\qed
\end{remark}

\begin{remark}
The function $h$ which appears in hypothesis (\ref{eq0.1} {\it i}) is an element of $H^{-1}(\Omega)$. Indeed, when $N\geq 3$, the exponent $ r=\frac{2N}{N+2}$ is nothing but the H\"older's conjugate $(2^*)'$ of the Sobolev's exponent $2^*$, i.e. 
\begin{equation}
\label{2333}
\mbox{when } N\geq 3,\quad\frac 1r=1-\frac{1}{2^*}, \mbox{ where } \frac{1}{2^*}=\frac 12-\frac 1N.
\end{equation}

Making an abuse of notation, we will set 
\begin{align}
\label{212bis}
\begin{cases}
2^*=\mbox{any }p \mbox{ with } 1<p<+\infty \mbox{ when } N=2,\\
2^*=+\infty \mbox{ when } N=1.
 \end{cases}
 \end{align}
With this abuse of notation, $h$ belongs to $L^r(\Omega)=L^{(2^*)'}(\Omega)\subset H^{-1}(\Omega)$ for all $N\geq 1$ since $\Omega$ is bounded.

This result is indeed a consequence of Sobolev's, Trudinger Moser's and Morrey's inequalities, which (with the above abuse of notation) assert that 
\begin{equation}
\label{2334}
\|v\|_{L^{2^*}(\Omega)}\leq C_S \|Dv\|_{(L^2(\Omega))^N}\quad \forall v\in H_0^1(\Omega) \mbox{ when } N\geq 1,
\end{equation}
where $C_S$ is a constant which depends only on $N$ when $N\geq 3$, which depends on $p$ and on $Q$ when $N=2$, and which depends on $Q$ when $N=1$, when $Q$ is any bounded open set such that $\Omega\subset Q$.
\qed
\end{remark}

\begin{remark}
Assumption \eqref{eq0.2}, namely the fact that for almost every $x\in\Omega$ the function $s\rightarrow F(x,s)$ is  nonincreasing in $s$, is not of the same nature as assumptions \eqref{car} and \eqref{eq0.1} on $F(x,s)$. 

We will only use assumption \eqref{eq0.2} when proving comparison and uniqueness results, namely Proposition~\ref{prop0} and Theorem~\ref{uniqueness}.
In contrast, all the others results of the present paper, and in particular the existence and stability results stated in Theorems~\ref{EUS} and \ref{est}, as well as the a priori estimates of Section~\ref{estimates}, do not use this assumption.
\qed
\end{remark}
\smallskip

\noindent {\bf 2.1. Notation}
\label{notation}
$\mbox{}$

We denote by $\mathcal{D}(\Omega)$ the space of the $C^\infty(\Omega)$ functions whose support is compact and included in $\Omega$, and by $\mathcal{D}'(\Omega)$ the space of distributions on $\Omega$.

We denote by $\mathcal{M}_b^+(\Omega)$ the space of nonnegative bounded Radon measures on $\Omega$.


Since $\Omega$ is bounded, $\|D w\|_{L^2(\Omega)^N}$ is a norm equivalent to $\|w\|_{H^1(\Omega)}$ on $H_0^1(\Omega)$. We set 
$$
\|w\|_{H_0^1(\Omega)}=\|D w\|_{(L^2(\Omega))^N}\, \,\,\forall w\in H_0^1(\Omega).
$$

For every $s\in \mathbb{R}$ and every $k>0$ we define as usual 
$$s^+=\max\{s,0\},\,\,s^-=\max\{0,-s\},$$
$$T_k(s)=\max\{-k,\min\{s,k\}\},\,\,\,\,\,G_k (s)= s- T_k (s).$$

For every measurable function  $l:x\in\Omega\rightarrow l(x)\in [0,+\infty]$ we denote
$$
\{l=0\}=\{x\in\Omega: l(x)=0\},\,\,\,\,\,\{l>0\}=\{x\in\Omega: l(x)>0\}.
$$

Finally, in the present paper, we denote by $\varphi$ and $\overline{\varphi}$ functions which belong to $H_0^1(\Omega)\cap L^\infty(\Omega)$, while we denote by $\phi$ and $\overline{\phi}$  functions which belong to $\mathcal{D}(\Omega)$.

\smallskip
\section{Definition of a solution to problem \eqref{eqprima}}
\label{secdefi}

\noindent{ }{\bf 3.1. The space $\mathcal V(\Omega)$ of test functions}
\label{secdefiV}
 $\mbox{}$

  In order to introduce the notion of solution to problem \eqref{eqprima} that we will use in the present paper, we define the following space $\mathcal V(\Omega)$ of test functions and a notation.
\begin{definition}\label{spazio}
The space $\mathcal V(\Omega)$ is the space of the functions $v$ which satisfy
\begin{equation}
             \label{vs}
            v\in H_0^1(\Omega)\cap L^\infty(\Omega), 
\end{equation}
\begin{equation}
\label{condv}
\begin{cases}
\exists I\,\text{finite},\, \exists  \hat{\varphi}_i, \exists \hat{ g}_i,i\in I,\exists  \hat{ f},\mbox{ with}\\
 \hat{\varphi}_i \in H_0^1(\Omega)\cap L^\infty(\Omega), \hat{ g}_i\in (L^2(\Omega))^N , \hat{ f}\in L^1(\Omega),\\ \mbox{such that }
  \displaystyle-div\, {}^t\!A(x)Dv=\sum_{i \in I} \hat{\varphi}_i (-div \,\hat{ g}_i)+\hat{ f} \,\, \mbox{in } \mathcal{D}'(\Omega).
 \end{cases}
 \end{equation}
 \qed
\end{definition} 

 In the definition of $\mathcal{V}(\Omega)$ we use the notation $\hat{\varphi_i}$, $\hat{g_i}$, and $\hat{f}$ to help the reader to identify the functions which enter in the definition of the functions of $\mathcal{V}(\Omega)$.

Note that $\mathcal{V}(\Omega)$ is a vector space.


\begin{definition}
\label{def32}
 When  $v\in \mathcal V(\Omega)$ with $$\displaystyle-div\, {}^t\!A(x)Dv=\sum_{i \in I} \hat{\varphi}_i (-div \,\hat{ g}_i)+\hat{ f}\quad\mbox{in } \mathcal{D}'(\Omega),$$ 
 where $I$, $\hat{\varphi}_i$, $\hat{g}_i$ and $\hat{f}$ are as in \eqref{condv}, and when $y$ satisfies $$y\in H^1{\mbox{{\tiny loc}}}(\Omega)\cap L^\infty(\Omega) \mbox{ with } \varphi y\in H^1_0(\Omega)\, \forall\varphi\in H^1_0(\Omega)\cap L^\infty(\Omega), $$ we use the following notation: 
\begin{equation}\label{dc}
\langle\langle -div\, {}^t\!A(x)Dv, y \rangle\rangle_\Omega=  \sum_{i\in I} \displaystyle \int_\Omega \hat{ g}_i D(\hat{ \varphi}_i y)+\displaystyle \int_\Omega \hat{ f}  y \, \\.
\end{equation}
\qed
\end{definition}

\begin{remark}
\label{rem2}


 In  \eqref{condv},   the product  $ \hat{\varphi}_i (-div \,\hat{ g}_i)$ with  $\hat{\varphi}_i \in H_0^1(\Omega)\cap L^\infty(\Omega)$ and $\hat{ g}_i\in (L^2(\Omega))^N$ is, as usual, the distribution on $\Omega$ defined, for every $\phi\in \mathcal{D}(\Omega)$, as

\begin{equation}
\label{condv2} 
\langle \hat{\varphi}_i (-div \,\hat{ g}_i), \phi\rangle_{\mathcal D'(\Omega),\mathcal{D}(\Omega)} =\langle-div\, \hat{ g}_i, \,\hat{\varphi}_i \phi \rangle_{H^{-1}(\Omega),H^1_0(\Omega)}=\int_\Omega \hat{ g}_i D(\hat{ \varphi}_i \phi), 
\end{equation}
and the equality $\displaystyle-div\, {}^t\!A(x)Dv=\sum_{i \in I} \hat{\varphi}_i (-div \,\hat{ g}_i)+\hat{ f}$ holds in $\mathcal D'(\Omega)$.

  In notation \eqref{dc}, the right-hand side is correctly defined since $\hat{\varphi}_i y \in H^1_0(\Omega) $ and since $y\in L^\infty(\Omega)$. In contrast the left-hand side $\langle\langle -div\, {}^t\!A\,Dv, y \rangle\rangle_\Omega$ is just a notation. \qed
 \end{remark}
 \begin{remark}
 \label{rem34}
 If $y\in H^1_0(\Omega)\cap L^\infty(\Omega) $, then  $\varphi y\in H^1_0(\Omega)$ for every $\varphi \in H^1_0(\Omega)\cap L^\infty(\Omega)$, so that for every $v\in\mathcal V(\Omega)$, $\langle\langle -div\, {}^t\!A\,Dv, y \rangle\rangle_\Omega$ is defined. Let us prove that actually one has  
\begin{equation}\label{classic}
\begin{cases}
\, \forall v\in\mathcal{V}(\Omega), \, \forall y \in H_0^1(\Omega)\cap L^\infty(\Omega),\\
\langle\langle -div\, {}^t\!A(x)Dv, y \rangle\rangle_\Omega=\langle -div\, {}^t\!A(x)Dv, y\rangle_{H^{-1}(\Omega),H^1_0(\Omega) }. 
\end{cases}
\end{equation}
Indeed when $\displaystyle-div\, {}^t\!A(x)Dv=\sum_{i \in I} \hat{\varphi}_i (-div \,\hat{ g}_i)+\hat{ f}$, one has for every $\phi\in \mathcal{D}(\Omega)$ (see \eqref{dc} and \eqref{condv2})
\begin{eqnarray*}
\begin{cases}
\dys\langle\langle -div\, {}^t\!A(x)Dv, \phi \rangle\rangle_{\Omega}=\sum_{i\in I} \into \hat{g_i} D(\hat{\varphi_i} \phi)+\into \hat{f}\phi=\\
\dys= \langle\sum_{i\in I}  \hat{\varphi_i}(-div \hat{g_i})+\hat{f},\phi\rangle_{\mathcal D'(\Omega),\mathcal D(\Omega)}=\\
\dys=\langle -div\, {}^t\!A(x)Dv, \phi \rangle_{\mathcal D'(\Omega),\mathcal D(\Omega)}=\!\! \into\,\!\! {}^t\!A(x)DvD\phi,
\end{cases}
\end{eqnarray*}
and therefore
\begin{equation}
\label{condv4}
 \displaystyle\sum_{i\in I}\into \hat{g_i} D(\hat{\varphi_i} \phi)+\into \hat{f}\phi= \into \, {}^t\!A\, DvD\phi, \quad \forall\phi\in \mathcal D(\Omega).
\end{equation}
For every given function $y\in H_0^1(\Omega)\cap L^\infty(\Omega)$, taking a sequence $\phi_n\in \mathcal D(\Omega)$ such that 
\begin{equation*}
\phi_n\rightarrow y \quad \mbox{in } H_0^1(\Omega) \mbox{ strongly,}\,\, \phi_n\rightharpoonup y\,\, \mbox{in } L^\infty(\Omega) \mbox{ weakly-star},
\end{equation*}
and passing to the limit in \eqref{condv4} with $\phi=\phi_n$, we obtain \eqref{classic}. 
 \qed
\end{remark}
\begin{remark}
\label{examples}

 Let us give some examples of functions which belong to $\mathcal V(\Omega)$.
\smallskip
\begin{itemize}[leftmargin=*]
\item{$i)$} If  $\varphi_1, \varphi_2 \in H^1_0(\Omega)\cap L^\infty(\Omega)$, then $\varphi_1 \varphi_2 \in \mathcal V(\Omega)$. Indeed $\varphi_1\varphi_2 \in H^1_0(\Omega)\cap L^\infty(\Omega)$, and in the sense of distributions, one has
\begin{align}
                                \label{nrt}
-div\, {}^t\!A(x)D(\varphi_1\varphi_2)
=\hat\varphi_1(-div \,\hat{ g}_1)+\hat\varphi_2(-div \,\hat{ g}_2)+\hat{f}\, \mbox{ in } \mathcal{D}'(\Omega),                      
\end{align}                                
with $\hat\varphi_1=\varphi_1 \in H^1_0(\Omega)\cap L^\infty(\Omega)$, $\hat\varphi_2=\varphi_2 \in H^1_0(\Omega)\cap L^\infty(\Omega)$, $\hat{g_1}=\,{}^t\!A(x)D\varphi_2\in L^2(\Omega)^N$, $\hat{g_2}=\,{}^t\!A(x)D\varphi_1\in (L^2(\Omega))^N$, and $\hat{f}=-\,{}^t\!A(x)D\varphi_2D\varphi_1-\,{}^t\!A(x)D\varphi_1D\varphi_2\in L^1(\Omega)$.
\smallskip
\item{$ii)$} In particular, if $\varphi \in H^1_0(\Omega)\cap L^\infty(\Omega)$, then $\varphi^2 \in \mathcal V(\Omega)$, with
\begin{equation}
\label{condv5}
-div\, {}^t\!A(x)D\varphi^2=\hat\varphi (-div\, \hat{g})+\hat{f}\,\, {\rm in }\,\mathcal{D}'(\Omega),
\end{equation}
with $\hat\varphi=2\varphi \in H^1_0(\Omega)\cap L^\infty(\Omega)$, $\hat{g}=\, {}^t\!A(x)D\varphi\in (L^2(\Omega))^N$ and $\hat{f}=-2\, {}^t\!A(x)D\varphi D\varphi\in L^1(\Omega)$.
\smallskip
\item{$iii)$} If $\varphi\in H_0^1(\Omega)\cap L^\infty(\Omega)$ and has compact support, then $\varphi\in \mathcal V(\Omega)$.  Indeed
\begin{equation}
\label{39bis}
-div\, {}^t\!A(x)D\varphi=\overline{\phi}(-div\, {}^t\!A(x)D\varphi)\, \mbox{ in } \, \mathcal D'(\Omega),
\end{equation}
for every $\overline{\phi}\in\mathcal D(\Omega)$, with $\overline{\phi}= 1$ on the support of $\overline\phi$.
\smallskip
\item{$iv)$} 
In particular every $\phi\in \mathcal{D}(\Omega)$ belongs to $\mathcal V(\Omega)$.
\end{itemize}
\qed
\end{remark}
\smallskip

\noindent{ }{\bf 3.2. Definition of a solution to problem \eqref{eqprima}}
\label{sub32}
 $\mbox{}$
 
We now give the definition of a solution to problem \eqref{eqprima} that we will use in the present paper.
\begin{definition}\label{sol}
Assume that the matrix $A$ and the function $F$ satisfy \eqref{eq0.0}, \eqref{car} and \eqref{eq0.1}. We say that $u$ is a solution to problem \eqref{eqprima} if $u$ satisfies
\begin{equation}\label{sol1}
\begin{cases}
i)\, u\in L^2(\Omega)\cap H^1_{\mbox{{\tiny loc}}}(\Omega),\\
ii)\, u(x)\geq 0 \,\,\mbox{ a.e. } x\in\Omega,\\
iii)\, G_k(u)\in H_0^1(\Omega)\,\, \,\, \forall k>0,\\
iv)\, \varphi T_k(u)\in H^1_0(\Omega)\,\, \,\,\forall k>0,\,\, \forall \varphi \in H^1_0(\Omega)\cap L^\infty (\Omega),
\end{cases}
\end{equation}
\begin{align}\label{sol2}
\begin{cases}
\displaystyle\forall v \in \mathcal V(\Omega),\,v\geq 0,\\ \dys\mbox{with } -div \, {}^t\!A(x)Dv=\sum_{i \in I} \hat{ \varphi}_i (-div\, \hat{ g}_i)+\hat{ f} \mbox{ in } \mathcal{D}'(\Omega),\\
\mbox{where } \hat{\varphi_i}\in H_0^1(\Omega)\cap L^\infty(\Omega), \hat{g_i}\in (L^2(\Omega))^N, \hat{f}\in L^1(\Omega),\\
\mbox{one has}\\
i) \,\dys  \into F(x,u) v<+\infty,\\
ii) \,
\dys \into\, {}^t\!A(x)Dv DG_k(u)+  \displaystyle\sum_{i\in I} \into \hat{g_i} D(\hat{\varphi_i} T_k(u))+\into \hat{f} T_k(u)=\\=\langle -div\, {}^t\!A(x)Dv, G_k(u) \rangle_{H^{-1}(\Omega),H_0^1(\Omega)}+\langle\langle -div\, {}^t\!A(x)Dv, T_k(u)\rangle\rangle_{\Omega}=\\
\displaystyle =\into F(x,u) v \,\,\forall k>0.
\end{cases}
\end{align}
\qed
\end{definition}

\begin{remark}
\label{rem25}
\begin{itemize}[leftmargin=*]
\mbox{ }
\item{$i)$} Definition \ref{sol} is a mathematically correct framework which gives a meaning to the solution to problem \eqref{eqprima}; in contrast \eqref{eqprima} is only formal. 

In this Definition~\ref{sol}, the requirement \eqref{sol1} is the "space" (which is not a vectorial space) to which the solution should belong, while the requirement \eqref{sol2}, and specially (\ref{sol2} {\it ii}), expresses the partial differential equation in \eqref{eqprima} in terms of (non standard) test functions in the spirit of the solutions defined by transposition by J.-L. Lions and E. Magenes and by G. Stampacchia.
\smallskip

\item{$ii)$} Note that the statement (\ref{sol1} {\it iii}) formally contains the boundary condition ``$u=0$ on $\partial\Omega$". Indeed $G_k(u)\in H_0^1(\Omega)$ for every $k>0$ formally implies that ``$G_k(u)=0$ on $\partial\Omega$", i.e. ``$u\leq k$ on $\partial\Omega$" for every $k>0$, which implies ``$u=0$ on $\partial\Omega$" since $u\geq 0$ in $\Omega$.
\smallskip 
\item{$iii)$} In Section~\ref{estimates} below we will obtain a priori estimates for every solution $u$ to problem \eqref{eqprima} in the sense of Definition~\ref{sol} in the various ``spaces" which appear in \eqref{sol1}.

%
%
%
\smallskip 
\item{$iv)$} Note finally that (very) formally, one has
\begin{align*}
\begin{cases}
\dys``\langle -div\, {}^t\!A(x)Dv, G_k(u) \rangle_{H^{-1}(\Omega),H_0^1(\Omega)}=\int_{\Omega}(-div\, {}^t\!A(x)Dv) \,G_k(u)=\\
\dys=\int_{\Omega}v\, (-divA(x)DG_k(u))",
\end{cases}
\end{align*}
\begin{align*}
\begin{cases}
\dys ``\langle\langle -div\, {}^t\!A(x)Dv, T_k(u) \rangle\rangle_\Omega=\int_{\Omega}(-div\, {}^t\!A(x)Dv)\, T_k(u)=\\
\dys=\int_{\Omega}v\,(-divA(x)DT_k(u)",
\end{cases}
\end{align*}
so that (\ref{sol2} {\it ii}) formally means that
$$
``\int_{\Omega} v\, (-div\, {}^t\!A(x)Du)=\into F(x,u)v\, " \,\, \forall v\in \mathcal{V}(\Omega),\, v\geq0. 
$$
Since every $v$ can be written as $v=v^+-v^-$ with $v^+\geq 0$ and $v^-\geq 0$, one has formally 
(this is formal since we do not know whether $v^+$ and $v^-$ belong to 
$\mathcal{V}(\Omega)$ when $v$ belongs to $\mathcal{V}(\Omega)$)
$$
``-div\,A(x)Du=F(x,u)".
$$

Observe that the above formal computation has no meaning in general, while (\ref{sol2} {\it ii}) has a perfectly correct mathematical sense when $v\in\mathcal V(\Omega)$ and when $u$ satisfies \eqref{sol1}.
\qed
\end{itemize}

\end{remark}

The following Proposition~\ref{prop310} asserts that every solution to problem \eqref{eqprima} in the sense of Definition~\ref{sol} is a solution to \eqref{eqprima} in the sense of distributions. Note that Proposition~\ref{prop310} does not say anything about the boundary conditions satisfied by $u$ (see also Remark~\ref{rem25} {\it ii})).

\begin{proposition}
\label{prop310}
Assume that the matrix $A$ and the function $F$ satisfy \eqref{eq0.0}, \eqref{car} and \eqref{eq0.1}. Then for every solution to problem \eqref{eqprima} in the sense of Definition~\ref{sol} one has 
\begin{equation}
\label{3107}
u\geq 0 \mbox{ a.e. in } \Omega,\,u\in H^1_{\mbox{\tiny loc}}(\Omega),\, F(x,u)\in L^1_{\mbox{\tiny loc}}(\Omega),
\end{equation}
\begin{equation}
\label{3108}
-div\, A(x)Du=F(x,u)\mbox{ in } \mathcal{D}'(\Omega).
\end{equation}
\end{proposition}

\begin{proof}
Since every $\phi \in \mathcal{D}(\Omega)$ belongs to  $\mathcal{V}(\Omega)$ (see Remark~\ref{examples}~$iv)$), assumption (\ref{sol2} $i$) implies that
\begin{equation}
\label{31222}
\into F(x,u)\phi<+\infty, \quad \forall \phi\in \mathcal{D}(\Omega),\quad \phi\geq 0,
\end{equation}
which together with $F(x,u)\geq 0$ implies that $F(x,u)\in L^1_{\mbox{\tiny{loc}}}(\Omega)$.

Therefore \eqref{3107} holds true when $u$ is a solution to problem \eqref{eqprima} in the sense of Definition~\ref{sol}.
 
For what concerns \eqref{3108},  for every $\phi\in\mathcal{D}(\Omega)$ and $\overline{\phi}\in\mathcal{D}(\Omega)$ with $\overline{\phi}= 1$ on the support of $\phi$, one has  (see \eqref{39bis} with $\varphi=\phi$ as far as the second term is concerned)
\begin{eqnarray*}
\begin{cases}
\dys\langle -div\, {}^t\!A(x)D\phi, G_k(u) \rangle_{H^{-1}(\Omega),H_0^1(\Omega)}+\dys\langle\langle -div\, {}^t\!A(x)D\phi, T_k(u) \rangle\rangle_\Omega=\\
=\dys\int_{\Omega}\, {}^t\!A(x)D\phi \,DG_k( u)
+\dys\int_{\Omega}\, {}^t\!A(x)D\phi\, D(\overline{\phi} T_k(u))=\\
\dys =\int_{\Omega}\, {}^t\!A(x)D\phi Du= \into A(x)Du D\phi,
\end{cases}
\end{eqnarray*}
where we have used the fact that $u\in H^1_{\mbox{\tiny loc}}(\Omega)$ and that $\overline{\phi}= 1$ on the support of $\phi$. This implies that when $u$ is a solution to problem \eqref{eqprima} in the sense of Definition~\ref{sol}, one has 
\begin{equation}
\label{32222}
\into A(x)Du D\phi= \into F(x,u)\phi \quad \forall \phi\in\mathcal{D}(\Omega), \quad \phi\geq 0,
\end{equation}
namely 
$$
-div\, A(x) Du\geq F(x,u) \mbox{ in } \mathcal{D}'(\Omega).
$$
But \eqref{32222} also implies that
$$
\into A(x)Du D\phi= \into F(x,u)\phi, \quad \forall \phi \in\mathcal{D}(\Omega), \quad \phi\leq 0,
$$
namely 
$$
-div\, A(x) Du\leq F(x,u) \mbox{ in } \mathcal{D}'(\Omega).
$$

 This proves \eqref{3108}.
\end{proof}

\begin{remark}
\label{r314}
When $u$ satisfies \eqref{sol1}, then one has 
\begin{equation}
\label{3111}
\varphi Du\in (L^2(\Omega))^N\,\,\, \forall \varphi \in H_0^1(\Omega)\cap L^\infty(\Omega).
\end{equation}
More precisely, when $u$ satisfies (\ref{sol1}~{\it i}) and  (\ref{sol1}~{\it iii}), assertion  (\ref{sol1}~{\it iv}) is equivalent to \eqref{3111}.

Indeed, if $y\in H^1_{\mbox{\tiny loc}}(\Omega)$ and $\varphi \in H_0^1(\Omega)\cap L^\infty(\Omega)$, one has for every $k>0$
\begin{equation*}
\begin{cases}
\varphi Dy =\varphi DT_k(y)+\varphi DG_k(y) \mbox{ in }  (\mathcal D'(\Omega))^N,\\
D(\varphi T_k(y))=\varphi D T_k(y)+ T_k(y)D\varphi \mbox{ in } (\mathcal D'(\Omega))^N,
\end{cases}
\end{equation*}
and therefore
$$
\varphi Dy-D(\varphi T_k(y))=\varphi DG_k(y)-T_k(y)D\varphi \mbox{ in }  (\mathcal D'(\Omega))^N.
$$
This in particular implies that 
$$
\varphi Dy-D(\varphi T_k(y)) \mbox{ belongs to } (L^2(\Omega))^N
$$
when $y\in H^1_{\mbox{\tiny loc}}(\Omega)$, $G_k(y)\in H_0^1(\Omega)$ and $\varphi \in H_0^1(\Omega)\cap L^\infty(\Omega)$.

Therefore if  $y\in H^1_{\mbox{\tiny loc}}(\Omega)$, $G_k(y)\in H_0^1(\Omega)$ and $\varphi \in H_0^1(\Omega)\cap L^\infty(\Omega)$, one has 
$$
\varphi Dy\in (L^2(\Omega))^N\Longleftrightarrow D(\varphi T_k(y))\in (L^2(\Omega))^N,
$$
which in view of Lemma~\ref{lem0} below and of the inequality 
$
-\varphi k\leq \varphi T_k(u)\leq +\varphi k
$
implies, since $\varphi T_k(y)\in L^\infty(\Omega)\subset L^2(\Omega)$, that
$$
\varphi Dy\in (L^2(\Omega))^N\Longleftrightarrow \varphi T_k(y)\in H_0^1(\Omega).
$$

Taking $y=u$ gives the equivalence between (\ref{sol1} {\it iv}) and \eqref{3111}.
\qed
\end{remark}
\smallskip

\section{Statements of the existence, stability\\ and uniqueness results}
\label{stat}

In this section we state results of existence, stability and uniqueness of the solution to problem \eqref{eqprima} in the sense of Definition~\ref{sol}.

\begin{theorem}{\bf(Existence)}. 
  \label{EUS}
Assume that the matrix $A$ and the function $F$ satisfy \eqref{eq0.0}, \eqref{car} and \eqref{eq0.1}. Then there exists at least one solution $u$ to  problem \eqref{eqprima} in the sense of Definition~\ref{sol}. 
\end{theorem}
The proof of Theorem~\ref{EUS} will be done in Subsection~6.2. It relies on a stability result (see Theorem~\ref{est} below) and on a priori estimates of  $\|G_k(u)\|_{H_0^1(\Omega)}$ for every $k>0$, of $\|\varphi \,DT_k(u)\|_{(L^2(\Omega))^N}$ for every $\varphi\in H_0^1(\Omega)\cap L^\infty(\Omega)$ and of $\dys \int_{\{u\leq \delta\}}F(x,u)v$ for every $\delta>0$ and every $v\in \mathcal{V}(\Omega)$, $v\geq 0$, which are satisfied by every solution $u$ to problem \eqref{eqprima} in the sense of Definition~\ref{sol} (see Propositions~\ref{prop1}, \ref{prop2} and \ref{prop3}). 

  \begin{theorem}{\bf(Stability)}.
\label{est}
 Assume that the matrix $A$ satisfies assumption \eqref{eq0.0}. Let $F_n$ be a sequence of functions and $F_\infty$ be a function which all  satisfy assumptions
\eqref{car} and \eqref{eq0.1} for the same $h$ and the same $\Gamma$. Assume moreover that 
\begin{equation}
 \label{num1}
 \mbox{ a.e. } x\in \Omega, \, F_n(x,s_n)\rightarrow\! F_\infty(x,s_\infty) \mbox{ if } s_n\rightarrow s_{\infty}, s_n\geq 0,s_\infty\geq0.
 \end{equation} 
Let $u_n$ be any solution to problem \eqref{eqprima}$_n$ in the sense of Definition~\ref{sol}, where \eqref{eqprima}$_n$ is the problem \eqref{eqprima} with $F(x,s)$ replaced by $F_n(x,s)$. 

Then there exists a subsequence, still labelled by $n$, and a function $u_\infty$, which is a solution to problem  \eqref{eqprima}$_\infty$ in the sense of Definition~\ref{sol}, such that
\begin{align}
\label{num2}
\begin{cases}
\dys u_n\rightarrow u_\infty \mbox{ in } L^2(\Omega)\mbox{ strongly, in } H^1_{\mbox{\tiny loc}}(\Omega) \mbox{ strongly and a.e. in }\Omega,\\
\dys  G_k(u_n)\rightarrow  G_k(u_\infty) \, \mbox{ in } H_0^1(\Omega)\, \mbox{strongly } \forall k>0,\\
\dys \varphi T_k (u_n)\rightarrow \varphi T_k (u_{\infty}) \, \mbox{ in } H^1_0(\Omega)\, \mbox{strongly } \forall k>0,\,\,\forall \varphi\in H_0^1(\Omega)\cap L^\infty(\Omega).\\
\end{cases}
\end{align}
\end{theorem}

The proof of the Stability Theorem~\ref{est} will be done in Subsection~6.1.

Note that assumption \eqref{num1} is some type of uniform convergence of the functions $F_n(x,s)$ to the function $F_\infty(x,s)$, which is nevertheless non standard since this function can take the value $+\infty$ for $s=0$.

\bigskip 
 
  Finally, the following uniqueness result is an immediate consequence of the Comparison Principle stated and proved in Section~\ref{comparison} below. Note that both the Uniqueness Theorem~\ref{uniqueness} and the Comparison Principle of Proposition~\ref{prop0} are based on the nonincreasing character of the function $F(x,s)$ with respect to $s$.
 
 \begin{theorem}{\bf(Uniqueness)}. 
 \label{uniqueness}
 Assume that the matrix $A$ and the function $F$ satisfy \eqref{eq0.0}, \eqref{car} and \eqref{eq0.1}.  Assume moreover that the function $F(x,s)$ is nonincreasing with respect to $s$, i.e.  satisfies assumption \eqref{eq0.2}. Then the solution to problem \eqref{eqprima} in the sense of Definition~\ref{sol} is unique.
 
 \end{theorem}
 
%

 \begin{remark}
 When assumptions \eqref{eq0.0}, \eqref{car}, \eqref{eq0.1} as well as \eqref{eq0.2} hold true, Theorems~\ref{EUS}, \ref{est} and \ref{uniqueness} together  assert that problem \eqref{eqprima} is well posed in the sense of Hadamard in the framework of Definition~\ref{sol}.
 \qed
 \end{remark}
 
\smallskip

 \section{A  priori estimates}
 \label{estimates}

In this section we state and prove a priori estimates which are satisfied by every solution to problem \eqref{eqprima} in the sense of Definition~\ref{sol}.

\smallskip
\noindent{ }{\bf  5.1. A priori estimate of $G_k(u)$ in $H_0^1(\Omega)$}
\label{sub51}

\begin{proposition}{\bf (A priori estimate of $G_k(u)$ in $H_0^1(\Omega)$)}.
\label{prop1}
Assume that the matrix $A$ and the function $F$ satisfy \eqref{eq0.0}, \eqref{car} and \eqref{eq0.1}. Then for every $u$ solution to problem \eqref{eqprima} in the sense of Definition~\ref{sol}, one has  
\begin{align}
\label{num3}
\|G_k(u)\|_{H_0^1(\Omega)}=\|DG_k(u)\|_{(L^2(\Omega))^N}\leq \frac{C_S}{\alpha} \frac{ \|h\|_{L^r(\Omega)}}{\Gamma(k)}\,\, \forall k>0,
\end{align}
where $C_S$ is the (generalized) Sobolev's constant defined in \eqref{2334}.
\end{proposition}

\begin{remark}
\label{rem52u}
Since $\Omega$ is bounded, using Poincar\'e's inequality 
\begin{equation}
\label{53bis}
\|y\|_{L^2(\Omega)}\leq C_P(\Omega)\|Dy\|_{(L^2(\Omega))^N}\quad \forall y\in H_0^1(\Omega),
\end{equation}
and writing $u=T_k(u)+G_k(u)$, one easily deduces from \eqref{num3} that every solution $u$ to problem \eqref{eqprima} in the sense of Definition~\ref{sol} satisfies the following a priori estimate in $L^2(\Omega)$
\begin{align}
\label{53ter}
\|u\|_{L^2(\Omega)}\dys\leq k|\Omega|^{\frac 12}+C_P(\Omega)\frac{C_S}{\alpha} \frac{ \|h\|_{L^r(\Omega)}}{\Gamma(k)}\,\, \forall k>0,
\end{align}
which, taking $k=k_0$ for some $k_0$ fixed or minimizing the right-hand side in $k,$ provides an a priori estimate of $\dys\|u \|_{L^2(\Omega)}$ which does not depend on $k$; unfortunately minimizing in $k$ does not give an explicit constant for a general function $\Gamma$. 
\qed
\end{remark}

\begin{remark}{\bf (Formal proof of Proposition~\ref{prop1}).}
\label{Test}
Taking formally $G_k(u)$ as test function in \eqref{eqprima} one obtains 
\begin{equation}
                              \label{condfc3}
\into A(x) DG_k(u) D G_k(u)=\into F(x,u)G_k(u) \,\,  \forall k>0.
\end{equation}

 Estimate \eqref{num3} then follows from the growth condition (\ref{eq0.1} {\it iii}) and from the facts that $G_k(u)=0$ on the set $\{u\leq k\}$, and that $\Gamma(s)$ is nondecreasing, so that 
$$
0\leq F(x,u)G_k(u)\leq  \frac{h(x)}{\Gamma(u)}G_k(u)\leq  \frac{h(x)}{\Gamma(k)}G_k(u)\,\mbox{ a.e. } x\in\Omega,
$$
and finally from the (generalized) Sobolev's inequality \eqref{2334}.

This formal computation will be made mathematically correct in the proof below.
\qed
\end{remark}

\begin{proof}[{\bf Proof of Proposition~\ref{prop1}.}]
$\mbox{}$

Define for every $k$ and $n$ with $0<k<n$ the function $S_{k,n}$ as
\begin{equation}
\label{defs}
S_{k,n}(s)=\begin{cases}
 0 & \mbox{if } 0\leq s\leq k, \\
 s-k & \mbox{if }  k\leq s\leq n,\\
n-k & \mbox{if } n\leq s.
\end{cases}
\end{equation}
\smallskip
\noindent{{{\bf First step}}}. In this step we will prove that 
\begin{equation}
\label{sev2}
S_{k,n}(u)\in \mathcal{V}(\Omega).
\end{equation}

Observe that  for every $n>k$ we have
$$0\leq S_{k,n}(u)\leq G_k(u) \mbox{ and } |D S_{k,n}(u)|=\chi_{_{\{k\leq u\leq n\}}} |Du|\leq |D G_k(u)|.$$ 

\noindent By (\ref{sol1} {\it iii}) and Lemma~\ref{lem0} of Appendix~\ref{appendixa} below this implies that
\begin{equation}
\label{sev}
S_{k,n}(u)\in H_0^1(\Omega)\cap L^\infty(\Omega).
\end{equation}

Let us prove that \eqref{sev2} holds true.

Let  $\psi_k:\mathbb{R}^+\rightarrow \mathbb{R}$ be a $C^1$ nondecreasing function such that 
$$
\psi_k(s)=0 \mbox{ for } 0\leq s\leq \frac k2 \mbox{ and } \psi_k(s)=1 \mbox{ for } s\geq k.
$$ 
Since $u\in H^1_{\mbox{\tiny loc}}(\Omega)$ one has
\begin{eqnarray*}
D\psi_k(u)=\psi_k'(u)Du=\psi_k'(u)DG_{\frac k2}(u) \, \mbox{ in } \mathcal{D}'(\Omega),
\end{eqnarray*}
from which using (\ref{sol1} {\it iii}) we deduce  that $D\psi_k(u)\in (L^2(\Omega))^N$. Therefore $\psi_k(u)\in H^1(\Omega)\cap L^\infty(\Omega)$, and then, using again (\ref{sol1} {\it iii}), the inequality
$$
0\leq \psi_k(u)\leq \frac{4}{k}G_{\frac k4}(u),
$$
(which results from $\frac{4}{k}G_{\frac k4}(s)\geq 1$ when $s\geq \frac k2$)  
and Lemma~\ref{lem0} below, we obtain
\begin{equation}
                            \label{num5}
                             \psi_k(u)\in H^1_0(\Omega)\cap L^\infty(\Omega).
\end{equation}
On the other hand, in view of \eqref{sev} one has  $-div\, {}^t\!A(x)D S_{k,n}(u)\in H^{-1}(\Omega)$, and one easily proves that 
\begin{equation}
\label{num7}
 -div\, {}^t\!A(x)DS_{k,n}(u)= \psi_k(u) (-div\, {}^t\!A(x)DS_{k,n}(u))\quad \mbox{in } \mathcal{D}'(\Omega),
\end{equation}
which implies \eqref{sev2} with $\hat{\varphi}=\psi_k(u)$, $\hat{g}= {}^t\!A(x)\, DS_{k,n}(u)$ and $\hat{f}=0$.

\smallskip
\noindent{{{\bf Second step}}}. Since $S_{k,n}(u)\in\mathcal{V}(\Omega)$ by \eqref{sev2} and since $S_{k,n}(u)\geq 0$, we can use $S_{k,n}(u)$ as test function in  (\ref{sol2} {\it ii}). We get, using \eqref{num7},
\begin{align*}
\begin{cases}
\dys\into  {}^t\!A(x) D S_{k,n}(u)D G_k(u)+\into  {}^t\!A(x) D S_{k,n}(u) D(\psi_k(u)T_k(u))=\\
\dys =\langle-div \, {}^t\!A(x)DS_{k,n}(u),G_k(u)\rangle_{H^{-1}(\Omega),H_0^1(\Omega)}+\\\dys+\langle\langle-div \, {}^t\!A(x)DS_{k,n}(u),T_k(u)\rangle\rangle_{\Omega}=\\
\dys=\into F(x,u) S_{k,n}(u),
\end{cases}
\end{align*}
where the second term of the left-hand side is zero since 
$$
D(\psi_k(u)T_k(u))=\psi_k(u)D T_k(u)+T_k(u)D\psi_k(u) \mbox{ in } \mathcal{D}'(\Omega),
$$
and since $D S_{k,n}(u)=0$ on $\{u\leq k\}$, while $DT_k(u)=0$ and $D\psi_k(u)=0$ on $\{u\geq k\}$. This gives
\begin{equation}
\label{sev3}
\into A(x)DG_k(u)DS_{k,n}(u)=\into F(x,u)S_{k,n}(u).
\end{equation}

\smallskip
\noindent{{{\bf Third step}}}.  Since $S_{k,n}(s)=0$ for $s\leq k$, one has, using the growth condition (\ref{eq0.1} {\it iii}) and the (generalized) Sobolev's inequality \eqref{2334} 
\begin{align}
\label{512bis}
\begin{cases}\vspace{0.1cm}
\dys \into F(x,u) S_{k,n}(u)\leq \into \frac{h(x)}{\Gamma(u)}\chi_{_{\{u>k\}}}S_{k,n}(u) \leq \into \frac{h(x)}{\Gamma(k)}S_{k,n}(u)\leq\\
\dys\leq \frac{\|h\|_{L^r((\Omega)}}{\Gamma(k)}\|S_{k,n}(u)\|_{L^{2^*}(\Omega)}
\dys \leq \frac{\|h\|_{L^r((\Omega)}}{\Gamma(k)} C_S\|D S_{k,n}(u)\|_{(L^2(\Omega))^N}\,\\ \forall k>0,\forall n>k.
\end{cases}
\end{align}

With the coercivity \eqref{eq0.0} of the matrix $A$, \eqref{sev3} and \eqref{512bis} imply that 
\begin{equation}
\label{512ter}
\alpha \|D S_{k,n}(u) \|_{(L^2(\Omega))^N}\leq C_S  \frac{\|h\|_{L^r((\Omega)}}{\Gamma(k)} \, \forall k>0,\forall n>k.
\end{equation}

Therefore $S_{k,n}(u)$ is bounded in $H_0^1(\Omega)$ for $k>0$ fixed independently of $n>k$, and the left-hand side (and therefore the right-hand side)  of \eqref{sev3} is bounded independently of $n>k$. Since
\begin{equation}
\label{512quato}
S_{k,n}(u) \rightharpoonup G_k(u) \mbox{ in } H_0^1(\Omega) \mbox{ weakly and a.e. } x\in\Omega \mbox{ as } n\to+\infty,
\end{equation}
applying Fatou's Lemma to the right-hand side of \eqref{sev3} one deduces that 
\begin{equation}
\label{50}
\into F(x,u)G_k(u)<+\infty\,\, \forall k>0.
\end{equation}

Then using \eqref{512quato}  in the left-hand side and Lebesgue's dominated convergence in the right-hand side of \eqref{sev3}, one obtains \eqref{condfc3}.

Estimate \eqref{num3} follows either from \eqref{512ter} and \eqref{512quato} or from \eqref{condfc3}.
\end{proof}
\smallskip

\noindent{ }{\bf 5.2. A priori estimate of $\varphi DT_k(u)$ in $(L^2(\Omega))^N$ for $\varphi \in H_0^1(\Omega)\cap L^\infty(\Omega)$}
\label{sub52}
\begin{proposition}
\label{prop2}
Assume that the matrix $A$ and the function $F$ satisfy \eqref{eq0.0}, \eqref{car} and \eqref{eq0.1}. Then for every $u$ solution to problem \eqref{eqprima} in the sense of Definition~\ref{sol} one has
\begin{align}
\label{619bis}
\begin{cases}
\dys\|\varphi DT_k(u)\|^2_{(L^2(\Omega))^N}\dys\leq \\\dys\leq\frac{32 k^2} {\alpha^2} \|A\|_{(L^\infty(\Omega))^{N\times N}}^2\|D\varphi\|_{(L^2(\Omega))^N}^2+\frac{C_S^2}{\alpha^2}\frac{\|h\|^2_{L^{r}(\Omega)}}{\Gamma(k)^2}\|\varphi\|_{L^\infty(\Omega)}^2\\
\dys \forall k>0,\,\, \forall\varphi\in H_0^1(\Omega)\cap L^\infty(\Omega),
\end{cases}
\end{align}
where $C_S$ is the (generalized) Sobolev's constant  defined in \eqref{2334}.
\end{proposition}
\begin{remark}
\label{54bis}
From the a priori estimate \eqref{619bis} and from $D(\varphi T_k(u))=\varphi DT_k(u)+T_k(u)D\varphi$, one deduces that every solution $u$ to problem \eqref{eqprima} in the sense of Definition~\ref{sol} satisfies the following a priori estimate of  $\varphi T_k(u)$ in $H_0^1(\Omega)$
\begin{align}
\label{516bis}
\begin{cases}
\dys\|\varphi T_k(u)\|^2_{H_0^1(\Omega)}=\|D(\varphi T_k(u))\|^2_{(L^2(\Omega))^N}\leq\\
\dys \leq \left(\frac{64 k^2} {\alpha^2} \|A\|_{(L^\infty(\Omega))^{N\times N}}^2+2k^2\right) \|D\varphi\|^2_{(L^2(\Omega))^N} 
\dys + 2\frac{C_S^2}{\alpha^2}\frac{\|h\|^2_{L^{r}(\Omega)}}{\Gamma(k)^2}\|\varphi\|_{L^\infty(\Omega)}^2\\
\dys \forall k>0,\,\, \forall\varphi\in H_0^1(\Omega)\cap L^\infty(\Omega).
\end{cases}
\end{align}
\qed
\end{remark}

\begin{remark}
\label{99}
Using estimates \eqref{prop1} of $DG_k(u)$ and \eqref{619bis} of $\varphi DT_k(u)$ in $(L^2(\Omega))^N,$  as well as 
$\varphi Du=\varphi DT_k(u)+\varphi DG_k(u)$, and $|DT_k(u)| |DG_k(u )|=0$ almost everywhere in $\Omega$,
one deduces that every solution $u$ to problem \eqref{eqprima} in the sense of Definition~\ref{sol} satisfies the following a priori estimate of  $\varphi Du$ in $(L^2(\Omega))^N$
\begin{align}
\label{111}
\begin{cases}
\dys\|\varphi Du \|^2_{(L^2(\Omega))^N}\dys=\dys\|\varphi DT_k(u) \|^2_{(L^2(\Omega))^N}+\dys\|\varphi DG_k(u) \|^2_{(L^2(\Omega))^N}\leq \\\leq\dys \frac{32 k^2} {\alpha^2} \|A\|_{(L^\infty(\Omega))^{N\times N}}^2\|D\varphi\|_{(L^2(\Omega))^N}^2+2\frac{C_S^2}{\alpha^2}\frac{\|h\|^2_{L^{r}(\Omega)}}{\Gamma(k)^2}\|\varphi\|_{L^\infty(\Omega)}^2\\
\dys \forall k>0,\,\, \forall\varphi\in H_0^1(\Omega)\cap L^\infty(\Omega),
\end{cases}
\end{align}  
which, taking $k=k_0$ for some $k_0$ fixed or minimizing the right-hand side in $k,$ provides an a priori estimate of $\dys\|\varphi Du \|^2_{(L^2(\Omega))^N}$ which does not depend on $k$; unfortunately minimizing in $k$ does not give an explicit constant for a general function $\Gamma$. 

\qed
\end{remark}
\begin{remark}
\label{55bis}
 Since for every $\phi\in \mathcal{D}(\Omega)$  one has $
D(\phi u)=\phi Du+(T_k(u)+G_k(u))D\phi$, 
which implies that $
|D(\phi u)|\leq |\phi Du|+k|D\phi|+\|D\phi\|_{(L^\infty(\Omega))^N}|G_k(u)|
$, and then using the inequality 
$
(a+b+c)^2\leq 3(a^2+b^2+c^2)
$
and the a priori estimates \eqref{111} and \eqref{num3} together with Poincar\'e's inequality \eqref{53bis}, one deduces that every solution $u$ to problem \eqref{eqprima} in the sense of Definition~\ref{sol} satisfies the following a priori estimate of $u$ in $H^1_{\mbox{\tiny loc}}(\Omega)$
\begin{align}
\label{512bis1}
\begin{cases}
\|\phi u\|^2_{H_0^1(\Omega)}=\|D(\phi u)\|^2_{(L^2(\Omega))^N}\leq\\
\leq \dys3\left(\frac{32 k^2} {\alpha^2} \|A\|_{(L^\infty(\Omega))^{N\times N}}^2\|D\phi\|_{(L^2(\Omega))^N}^2+2\frac{C_S^2}{\alpha^2}\frac{\|h\|^2_{L^{r}(\Omega)}}{\Gamma(k)^2}\|\phi\|_{L^\infty(\Omega)}^2+\right.\\
\dys \left. + k^2\|D\phi\|^2_{L^2(\Omega)}+C_P^2(\Omega)\frac{C_S^2}{\alpha^2}\frac{\|h\|^2_{L^r(\Omega)}}{\Gamma(k)^2}\|D\phi\|^2_{(L^\infty(\Omega))^N}\right)\\
\forall k>0,\,\,\,\dys \forall \phi\in \mathcal{D}(\Omega),
\end{cases}
\end{align}
which, taking $k=k_0$ for some $k_0$ fixed or minimizing the right-hand side in $k,$ provides an a priori estimate of $\dys\|\phi u \|^2_{H_0^1(\Omega)}$ for every fixed $\phi\in \mathcal{D}(\Omega)$, i.e. an a priori estimate of $\|u\|^2_{H^1_{\mbox{\tiny loc}}(\Omega)}$, which does not depend on $k$.
\qed
\end{remark}

\begin{remark}{\bf (Formal proof of Proposition~\ref{prop2}).}
\label{57bis}
The computation that  we will perform in the present Remark is formal. We will make it mathematically correct in the proof below.

The idea of the proof of Proposition~\ref{prop2} is to formally use $\varphi^2 T_k(u)$ as test function in \eqref{eqprima}, where $\varphi\in H_0^1(\Omega)\cap L^\infty(\Omega)$. We formally get 
\begin{equation}
\label{36}
\into A(x)DuDT_k(u)\,\varphi^2 + 2\into A(x)Du D\varphi\, \varphi T_k(u)=\into F(x,u)\varphi^2 T_k(u),
\end{equation}
in the second term of which we write $Du=DT_k(u)+DG_k(u)$. Using the coercivity \eqref{eq0.0} of the matrix $A$ we have 
\begin{align}
\label{37}
\begin{cases}\vspace{0.1cm}
\dys\alpha \into \varphi^2|DT_k(u)|^2\leq\\ 
\dys\leq 2\left|\into A(x)DT_k(u) D\varphi\,\varphi T_k(u)\right|
+2 \dys \left|\into A(x)DG_k(u) D\varphi\,\varphi T_k(u)\right|+\\
\dys+ \into F(x,u)\varphi^2 T_k(u) \quad \forall \varphi \in H_0^1(\Omega)\cap L^\infty(\Omega).
\end{cases}
\end{align}

In this inequality we will use the estimate
\begin{align}\label{67}
\begin{cases}
\dys 2\left|\into A(x)DT_k(u) D\varphi\,\varphi T_k(u)\right|+ 2\left|\into A(x)DG_k(u) D\varphi\,\varphi T_k(u)\right|\leq\\ \dys\leq 2k \|A\|_{(L^\infty(\Omega))^{N\times N}}\|D\varphi\|_{(L^2(\Omega))^N}\|\varphi DT_k(u)\|_{(L^2(\Omega))^N}+\\
\dys+ 2k \|A\|_{(L^\infty(\Omega))^{N\times N}}\|D\varphi\|_{(L^2(\Omega))^N}\|\varphi\|_{L^\infty(\Omega)}\|DG_k(u)\|_{(L^2(\Omega))^N}.
\end{cases}
\end{align}

On the other hand, since $0\leq T_k(u)\leq k$, and using formally $\varphi^2$ as test function in \eqref{eqprima}, we have
\begin{align}\label{37bis}
\begin{cases}\vspace{0.1cm}
0\leq \dys\into F(x,u)\varphi^2T_k(u)\leq k\dys\into F(x,u) \varphi^2=
\\=k \dys\into A(x)Du D\varphi^2
 =2k\dys\into A(x) Du D\varphi \,\varphi=\\=\dys 2k\into A(x)DT_k(u) D\varphi\,\varphi +
2k \dys \into A(x)DG_k(u) D\varphi\,\varphi \leq \\
\dys\leq 2k \|A\|_{(L^\infty(\Omega))^{N\times N}}\|D\varphi\|_{(L^2(\Omega))^N}\|\varphi DT_k(u)\|_{(L^2(\Omega))^N}+\\
+ 2k \|A\|_{(L^\infty(\Omega))^{N\times N}}\|D\varphi\|_{(L^2(\Omega))^N}\|\varphi\|_{L^\infty(\Omega)}\|DG_k(u)\|_{(L^2(\Omega))^N}.
\end{cases}
\end{align}

Collecting together \eqref{37}, \eqref{67} and \eqref{37bis} we obtain
\begin{align}\label{38}
\begin{cases}
\dys\alpha \into \varphi^2|DT_k(u)|^2\leq\\
\leq\dys 4 k\|A\|_{(L^\infty(\Omega))^{N\times N}}\|D\varphi\|_{(L^2(\Omega))^N}\|\varphi DT_k(u)\|_{(L^2(\Omega))^N}+\\
+ 4k \|A\|_{(L^\infty(\Omega))^{N\times N}}\|D\varphi\|_{(L^2(\Omega))^N}\|\varphi\|_{L^\infty(\Omega)}\|DG_k(u)\|_{(L^2(\Omega))^N}\\
\dys \forall k>0,\,\, \forall\varphi\in H_0^1(\Omega)\cap L^\infty(\Omega).
\end{cases}
\end{align}

Using Young's inequality in the first term of the right-hand side of \eqref{38} and the estimate \eqref{prop1} of $\|DG_k(u)\|_{(L^2(\Omega))^N}$ in the second term provides an estimate of $\|\varphi DT_k(u)\|^2_{(L^2(\Omega))^N}$.
 \qed
\end{remark}

\begin{proof}[{\bf Proof of Proposition~\ref{prop2}}]
$\mbox{}$

\smallskip

\noindent{{{\bf First step}}}.
By (\ref{sol1} {\it iv}) we have $\varphi T_k(u)\in H_0^1(\Omega)\cap L^\infty(\Omega)$, which implies, using Remark~\ref{examples} {\it i}), that 
\begin{equation}
\label{524ter}
\varphi^2 T_k(u)=\varphi\,\varphi T_k(u)\in\mathcal{V}(\Omega),
\end{equation}
with
\begin{align*}
\begin{cases}
\dys-div\, {}^t\!A(x)D(\varphi^2T_k(u))=\\
\dys=\varphi(-div \, {}^t\!A(x)D(\varphi T_k(u)))-\, {}^t\!A(x)D(\varphi T_k(u))D\varphi \,+\\
+(\varphi T_k(u))(-div \, {}^t\!A(x)D\varphi)-\, {}^t\!A(x)D\varphi D(\varphi T_k(u)) \quad \mbox{ in } \mathcal{D}'(\Omega).
\end{cases}                            
\end{align*}                           

Since $\varphi^2 T_k(u) \in \mathcal{V}(\Omega)$ with $\varphi^2 T_k(u)\geq 0$ we can use $v=\varphi^2 T_k(u)$ as test function in (\ref{sol2} {\it ii}). Denoting by $j$ the value of $k$ which appears in (\ref{sol2} {\it ii}), we get
\begin{align}
\label{nlt}
\begin{cases}
\dys\into  {}^t\!A(x)D(\varphi^2 T_k(u))DG_j(u)\,+\\\vspace{0.1cm}
\dys+ \into  {}^t\!A(x)D(\varphi T_k(u))D(\varphi T_j(u))-\into  {}^t\!A(x)D(\varphi T_k(u))D\varphi \,T_j(u)\,+\\
\dys+ \into  {}^t\!A(x)D\varphi\,D(\varphi T_k(u) T_j(u))-\into  {}^t\!A(x)D\varphi D(\varphi  T_k(u)) \,T_j(u)=\\
\dys= \langle-div \, {}^t\!A(x)D(\varphi^2 T_k(u) ),G_j(u)\rangle_{H^{-1}(\Omega),H_0^1(\Omega)}+\\\dys+\langle\langle-div \, {}^t\!A(x)D(\varphi^2 T_k(u) ),T_j(u)\rangle\rangle_{\Omega}=\\
\dys= \into F(x,u)\varphi^2 T_k(u)\quad  \forall\varphi\in H_0^1(\Omega)\cap L^\infty(\Omega),\quad \forall j>0,
\end{cases}
\end{align}
where we observe that the fourth term of the left-hand side makes sense due to the fact that $\varphi T_k(u)\in H^1_0(\Omega)\cap L^\infty(\Omega)$ by  (\ref{sol1} {\it iv}), which implies that $(\varphi T_k(u)) T_j(u)\in H^1_0(\Omega)\cap L^\infty(\Omega)$ again by (\ref{sol1} {\it iv}).
 
Since $u\in H^1_{\mbox{\tiny loc}}(\Omega)$, we can expand in $L^1_{\mbox{\tiny loc}}(\Omega)$ the integrands of each of the $5$ terms of the left-hand side of \eqref{nlt}. One obtains $13$ terms whose integrands belong to $L^1(\Omega)$ since $DG_j(u)$, $\varphi DT_k(u)$ and $\varphi DT_j(u)$ belong to $(L^2(\Omega))^N$. A simple but tedious computation leads from \eqref{nlt} to 
\begin{align}
\label{num13}
\begin{cases}
\dys \into  A(x) DT_j(u)DT_k(u)\varphi^2+\,\dys \into  A(x) DG_j(u)DT_k(u)\ \varphi^2+\\
+2\dys\into A(x) DT_j(u)D\varphi\,\varphi T_k(u)+2\dys\into A(x) DG_j(u)D\varphi\,\varphi T_k(u)=\\
\dys=\into F(x,u)\varphi^2T_k(u) \quad
\dys\forall \varphi\in H_0^1(\Omega)\cap L^\infty(\Omega),\,\, \forall j>0.
\end{cases}
\end{align}
Taking $j=k$ gives, since $|DG_k(u)||DT_k(u)|=0$ almost everywhere in $\Omega,$ 
\begin{align}\label{num14}
\begin{cases}
\dys \into  A(x) DT_k(u)DT_k(u)\varphi^2\,+\\
\dys+2\into A(x) DT_k(u)D\varphi\,\varphi T_k(u)+2\dys\into A(x) DG_k(u)D\varphi\,\varphi T_k(u)=\\
\dys= \into F(x,u)\varphi^2T_k(u) \quad \forall \varphi\in H_0^1(\Omega)\cap L^\infty(\Omega).
\end{cases}
\end{align}
This result is nothing but \eqref{36}, which had been formally obtained in Remark~\ref{57bis} by taking $\varphi^2T_k(u)$ as test function in \eqref{eqprima}, but the proof that we just performed is  mathematically correct. From \eqref{num14} and the coercivity \eqref{eq0.0} of the matrix $A$ we deduce that 
\begin{align}
\label{524bis}
\begin{cases}
\dys\alpha \into \varphi^2|DT_k(u)|^2\leq\\
\dys \leq 2\left| \dys\into A(x)DT_k(u)D\varphi\,\varphi T_k(u)\right|+2\left| \dys\into A(x)DG_k(u)D\varphi\,\varphi T_k(u)\right|+\\
\dys + \into F(x,u)\varphi^2T_k(u) \quad \forall \varphi\in H_0^1(\Omega)\cap L^\infty(\Omega),
\end{cases}
\end{align}
which is nothing but \eqref{37} of the formal proof made in Remark~\ref{57bis}.

\smallskip
\noindent{{{\bf Second step}}}. 
 As far as the first and the second terms of the right-hand side of \eqref{524bis} are concerned, we have, as in \eqref{67}, 
\begin{align}\label{67bis},
\begin{cases}
\dys 2\left|\into A(x)DT_k(u) D\varphi\,\varphi T_k(u)\right|+2\left|\into A(x)DG_k(u) D\varphi\,\varphi T_k(u)\right|\leq\\ \dys\leq 2k \|A\|_{(L^\infty(\Omega))^{N\times N}}\|D\varphi\|_{(L^2(\Omega))^N}\|\varphi DT_k(u)\|_{(L^2(\Omega))^N}+\\
\dys +2k \|A\|_{(L^\infty(\Omega))^{N\times N}}\|D\varphi\|_{(L^2(\Omega))^N}\|\varphi\|_{L^\infty(\Omega)}\|DG_k(u)\|_{(L^2(\Omega))^N}.
\end{cases}
\end{align}

On the other hand, since for every $\varphi\in H_0^1(\Omega)\cap L^\infty(\Omega)$, $\varphi^2$ belongs to $\mathcal{V}(\Omega)$ (see Remark~\ref{examples} {\it ii})) and since  $\varphi^2\geq0$, using $\varphi^2$ as test function in (\ref{sol2} {\it ii})  gives 
\begin{align*}
\begin{cases}
\langle-div \, {}^t\!A(x)D\varphi^2 ,G_k(u)\rangle_{H^{-1}(\Omega),H_0^1(\Omega)}+\langle\langle -div\, {}^t\!A(x)D\varphi^2,T_k(u) \rangle\rangle_\Omega=\\\dys= \into F(x,u)\varphi^2,
\end{cases}
\end{align*}
which gives 
\begin{align}
\label{555}
\begin{cases}
\dys2\into  A(x)DG_k(u) D\varphi\,\varphi+\dys2\into  A(x)DT_k(u) D\varphi\,\varphi=\\
\dys=\into F(x,u)\varphi^2\quad \forall\varphi\in H_0^1(\Omega)\cap L^\infty(\Omega).
\end{cases}
\end{align}

Therefore we have, for the last term of the right-hand side of \eqref{524bis}, 
\begin{align}
\label{528bis}
\begin{cases}\vspace{0.1cm}
\dys0\leq  \into F(x,u)\varphi^2 T_k(u)\leq  k\into F(x,u)\varphi^2 =\\
\dys = 2k \into A(x)DT_k(u) D\varphi\,\varphi +2k \into A(x)DG_k(u) D\varphi\,\varphi \leq\\  \leq2k \|A\|_{(L^\infty(\Omega))^{N\times N}} \|D\varphi\|_{(L^2(\Omega))^N}\|\varphi DT_k(u)\|_{(L^2(\Omega))^N}+\\+2k \|A\|_{(L^\infty(\Omega))^{N\times N}} \|D\varphi\|_{(L^2(\Omega))^N}\|\varphi\|_{L^\infty(\Omega)}\|DG_k(u)\|_{(L^2(\Omega))^N},
\end{cases}
\end{align}
which is nothing but \eqref{37bis} of the formal proof made in Remark~\ref{57bis}.\\

\smallskip

\noindent{{{\bf Third step}}}. Collecting together \eqref{524bis}, \eqref{67bis} and \eqref{528bis} we have proved that 
\begin{align}\label{38bis}
\begin{cases}
\dys\alpha \into \varphi^2|DT_k(u)|^2\leq\\
\leq\dys 4 k\|A\|_{(L^\infty(\Omega))^{N\times N}}\|D\varphi\|_{(L^2(\Omega))^N}\|\varphi DT_k(u)\|_{(L^2(\Omega))^N}+\\
+ 4k \|A\|_{(L^\infty(\Omega))^{N\times N}}\|D\varphi\|_{(L^2(\Omega))^N}\|\varphi\|_{L^\infty(\Omega)}\|DG_k(u)\|_{(L^2(\Omega))^N}\\
\dys \forall k>0,\,\, \forall\varphi\in H_0^1(\Omega)\cap L^\infty(\Omega),
\end{cases}
\end{align}
which is nothing but \eqref{38} of the formal proof made in Remark~\ref{57bis}.

Using Young's inequality $\dys XY\leq \frac{\alpha}{2}X^2+\frac{1}{2\alpha}Y^2$ in the first term of the right-hand side of \eqref{38bis} and estimate \eqref{num3} of $\|DG_k(u)\|_{(L^2(\Omega))^N}$  in the second term gives
\begin{align}
\label{569}
\begin{cases}
\dys\frac{\alpha}{2}\|\varphi DT_k(u)\|^2_{(L^2(\Omega)^N}\leq\\\vspace{0.1cm}
\dys\leq  \frac {8k^2}{\alpha}\|A\|^2_{(L^\infty(\Omega))^{N\times N}} \|D \varphi\|^2_{(L^2(\Omega))^N}  +\\\vspace{0.1cm}
\dys + 4k\|A\|_{(L^\infty(\Omega))^{N\times N}}\|D \varphi\|_{(L^2(\Omega))^N} \|\varphi\|_{L^\infty(\Omega)} \frac{C_S}{\alpha}  \frac{\|h\|_{L^r(\Omega)}}{\Gamma(k)},
\end{cases}
\end{align}
in which we use again Young's inequality $\dys\frac{1}{\alpha}XY\leq \frac{1}{2\alpha}X^2+\frac{1}{2\alpha}Y^2$ in the last term of the right-hand side. This gives
\begin{align}
\label{570}
\begin{cases}
\dys\frac{\alpha}{2}\|\varphi DT_k(u)\|^2_{(L^2(\Omega)^N}\leq\\\vspace{0.1cm}
\dys\leq  \frac {16k^2}{\alpha}\|A\|^2_{(L^\infty(\Omega))^{N\times N}} \|D \varphi\|^2_{(L^2(\Omega))^N}  
\dys +\frac {C_S^2}{2\alpha} \frac{\|h\|^2_{L^r(\Omega)}}{\Gamma(k)^2}\|\varphi\|^2_{L^\infty(\Omega)},
\end{cases}
\end{align}
i.e. \eqref{619bis}. 

This completes the proof of Proposition~\ref{prop2}.
 \end{proof}
 
 \smallskip
\noindent{ } {\bf 5.3. Control of the integral $\dys\int_{\{u\leq \delta\}}F(x,u)v$}$\mbox{ }$
  \label{sub53}

 In this subsection we prove an a priori estimate (see \eqref{5701bis} below) which is a key point in the proofs of our results.

For $\delta>0$, we define the function $Z_\delta: s\in[0,+\infty[\rightarrow Z_\delta(s)\in[0,+\infty[$ by 
\begin{equation}
\label{num23bis}
Z_\delta(s)=\begin{cases}
1& \mbox{if } 0\leq s\leq \delta, \\
 -\frac{s}{\delta}+2 & \mbox{if }  \delta\leq s\leq 2\delta ,\\
  0 & \mbox{if } 2\delta\leq s.
\end{cases}
\end{equation}

\begin{proposition}
\label{prop3}
Assume that the matrix $A$ and the function $F$ satisfy \eqref{eq0.0}, \eqref{car} and \eqref{eq0.1}. Then for every $u$ solution to problem \eqref{eqprima} in the sense of Definition~\ref{sol} and for every $v$ such  that
\begin{equation}
\label{5701}
\begin{cases}
\displaystyle v \in \mathcal V(\Omega),\,v\geq 0,\\ \dys\mbox{with } -div \, {}^t\!A(x)Dv=\sum_{i \in I} \hat{ \varphi}_i (-div\, \hat{ g}_i)+\hat{ f} \mbox{ in } \mathcal{D}'(\Omega)\\
\mbox{where } \hat{\varphi_i}\in H_0^1(\Omega)\cap L^\infty(\Omega), \hat{g_i}\in L^2(\Omega)^N, \hat{f}\in L^1(\Omega),\end{cases}
\end{equation}
one has
\begin{align}
\label{5701bis}
\begin{cases}
\dys\forall \delta>0,\,\,\int_{\Omega}F(x,u)Z_\delta(u)v\leq\\
\dys\leq \frac 32 \left(\into \left|\sum_{i\in I} \hat{g_i} D\hat{\varphi_i}+ \hat{f}\right| \right)\delta+\sum_{i\in I} \into Z_\delta(u) \hat{g_i}Du \,\hat{\varphi_i}.
\end{cases}
\end{align}
\end{proposition}
\begin{remark}
\label{510bis}
Since $Z_\delta(s)\geq \chi_{_{\{u\geq \delta\}}}(s)$ for every $s\geq 0$, estimate \eqref{5701bis} provides an estimate of the integral $\dys\int_{\{u\leq \delta\}}F(x,u)v$ as announced in the title of Subsection~5.3. 
\qed
\end{remark}

\begin{remark}{\bf (Formal proof of Proposition~\ref{prop3}).} 
\label{rem538}
Estimate \eqref{5701bis} can be formally obtained by using in \eqref{eqprima} the test function $Z_\delta(u)v$, where $Z_\delta$ is defined by \eqref{num23bis} and where $v$ satisfies \eqref{5701}. One formally obtains
\begin{equation}
\label{36p}
\into F(x,u)Z_\delta(u)v=\into A(x) Du Du Z_\delta'(u)v+\into A(x) DuDv Z_\delta(u),
\end{equation}
and one observes that, denoting by $Y_\delta$ the primitive of  the function $Z_\delta$ (see \eqref{5702} below), one has,
\begin{align*}
\begin{cases}\vspace{0.1cm}
\dys \into A(x) Du Du Z_\delta'(u)v\leq 0,\\\vspace{0.1cm}
\dys \into A(x)Du Dv Z_\delta(u)=\into A(x)DY_\delta(u)Dv
\dys = \into  \, {}^t\!A(x)Dv DY_\delta(u)=\\
\dys = \sum_{i \in I} \into  \hat{ g}_i D(\hat{\varphi}_i Y_\delta(u))+\into \hat{f} Y_\delta(u)=\\
\dys=   \into  \left(  \sum_{i \in I}\hat{ g}_i D\hat{\varphi}_i + \hat{f} \right) Y_\delta(u)+ \sum_{i \in I} \into Z_\delta(u)\hat{g}_i Du\, \hat{\varphi}_i,
\end{cases}
\end{align*}
where one finally uses $\dys0\leq Y_\delta(s)\leq \frac 32 \delta$ to formally obtain estimate \eqref{5701bis}. 

These formal computations will be made mathematically correct in the proof below.
\qed
\end{remark}

As a consequence of Proposition~\ref{prop3} we have:
\begin{proposition}
\label{prop69}
Assume that the matrix $A$ and the function $F$ satisfy \eqref{eq0.0}, \eqref{car} and \eqref{eq0.1}. Then for every $u$ solution to problem \eqref{eqprima} in the sense of Definition~\ref{sol} one has
\begin{equation}
\label{5800}
\int_{\{u=0\}}F(x,u)v=0\quad \forall v\in\mathcal{V}(\Omega),v\geq 0.
\end{equation}
\end{proposition}
\begin{proof}[{\bf Proof of Proposition~\ref{prop69}}]
$\mbox{}$

Let $\delta>0$ tend to $0$ in \eqref{5701bis}. In the left-hand side one has
$$
\int_{\{u=0\}}F(x,u)v\leq \int_{\Omega} F(x,u)Z_\delta(u)v,
$$
while in the right-hand side $
Z_\delta(u)$ tends to $\chi_{_{\{u=0\}}}$ almost everywhere in $\Omega$; but since $u\in H^1_{\mbox{\tiny loc}}(\Omega)$, one has 
$
Du=0$ almost everywhere on the set $\{u=0\}$, and therefore, since $\hat{g_i}Du\,\hat{\varphi_i}\in L^1(\Omega)$, and since $0\leq Z_{\delta (u)}\leq 1$, one has, by Lebesgue's dominated convergence Theorem,
$$
\into Z_\delta(u)\hat{g_i}Du\,\hat{\varphi_i}\rightarrow 0\mbox{ as } \delta\to 0.
$$

This proves \eqref{5800}.
\end{proof}

\begin{proof}[{\bf Proof of Proposition~\ref{prop3}}]
$\mbox{}$

In addition to the function $Z_\delta(s)$ defined by \eqref{num23bis}, we define for $\delta>0$ the functions
\begin{equation}
\label{5702}
Y_\delta(s)=\begin{cases}
 s & \mbox{if } 0\leq s\leq \delta, \\
 -\frac{s^2}{2\delta}+2s-\frac{\delta}{2} & \mbox{if }  \delta\leq s\leq 2\delta ,\\
\frac{3}{2}\delta & \mbox{if } 2\delta\leq s,
\end{cases}
\end{equation}

\begin{equation}
\label{5702bis}
R_\delta(s)=\begin{cases}
 0 & \mbox{if } 0\leq s\leq \delta, \\
 \frac{s^2}{2\delta}-\frac{\delta}{2} & \mbox{if }  \delta\leq s\leq 2\delta ,\\
\frac{3}{2}\delta & \mbox{if } 2\delta\leq s.
\end{cases}
\end{equation}
Observe that $Z_\delta$, $Y_\delta$ and $R_\delta$ are Lipschitz continuous and piecewise $C^1$ functions with 
\begin{equation}
\label{5703}
\begin{cases}
\dys Y'_{\delta}(s)=Z_\delta(s), \,\,Z'_\delta(s)=-\frac{1}{\delta}\chi_{_{\{\delta< s< 2\delta\}}}(s),\,\, R'_{\delta}(s)=-sZ'_\delta(s),\\
\dys Y_\delta(s)=sZ_\delta(s)+R_\delta(s),\,\,0\leq Y_\delta(s)\leq \frac 32 \delta,\,\, \forall s\geq 0.
\end{cases}
\end{equation}
\\
\smallskip

\noindent{{{\bf First step}}}. In this first step we will prove that for every $\delta>0$ 
\begin{equation}
\label{5704}
\begin{cases}
Y_\delta(u)\in H^1_{\mbox{\tiny loc}}(\Omega)\cap L^\infty(\Omega),\\
\varphi Y_\delta(u)\in H_0^1(\Omega)\quad \forall\varphi\in H_0^1(\Omega)\cap L^\infty(\Omega),
\end{cases}
\end{equation}
\begin{equation}
\label{5705}
Z_\delta(u)v\in \mathcal{V}(\Omega)\quad \forall v\in\mathcal{V}(\Omega),
\end{equation}
\begin{equation}
\label{5706}
R_\delta(u)\in H_0^1(\Omega)\cap L^\infty(\Omega).
\end{equation}
\smallskip

Since $u\in H^1_{\mbox{\tiny loc}}(\Omega)$ by (\ref{sol1} {\it i}), one has $Y_\delta(u)\in H^1_{\mbox{\tiny loc}}(\Omega)\cap L^\infty(\Omega)$ with 
$$
DY_\delta(u)=Y'_\delta(u)Du=Z_\delta(u)Du\mbox{ in } \mathcal{D}'(\Omega),
$$
 and for every $\varphi\in H_0^1(\Omega)\cap L^\infty(\Omega)$, one has $\varphi Y_\delta(u)\in H_{\mbox{\tiny loc}}^1(\Omega)\cap L^\infty(\Omega)$ with
 $$
 D(\varphi Y_\delta(u))=\varphi Z_\delta(u)Du +Y_\delta(u)D\varphi \mbox{ in } \mathcal{D}'(\Omega),
 $$
 which using \eqref{3111} and the last assertion of \eqref{5703} implies that $D(\varphi Y_\delta(u))\in (L^2(\Omega))^N$, and therefore that $\varphi Y_\delta(u)\in H^1(\Omega)\cap L^\infty(\Omega)$. Since 
 $
 \dys 0\leq\varphi Y_\delta(u)\leq \frac 32 \delta \varphi
 $, and since $\varphi\in H_0^1(\Omega)$, Lemma~\ref{lem0} of Appendix~\ref{appendixa} below completes the proof of \eqref{5704}.
 \\
 \smallskip
 
 On the other hand  $Z_\delta(u)\in H^1_{\mbox{\tiny loc}}(\Omega)$ and 
 $$
 DZ_\delta(u)=-\frac{1}{\delta}\chi_{_{\{\delta< u< 2\delta\}}}Du \mbox{ in } \mathcal{D}'(\Omega).
 $$
 In view of (\ref{sol1} {\it iii}) with $k=\delta$, this implies that $DZ_\delta(u)\in (L^2(\Omega))^N$, and therefore 
 \begin{equation}
 \label{5707}
 Z_\delta(u)\in H^1(\Omega)\cap L^\infty(\Omega),
 \end{equation}
which in turn implies that
 \begin{equation}
 \label{548bis}
 Z_\delta(u)\varphi \in H_0^1(\Omega)\cap L^\infty(\Omega)\quad \forall \varphi \in H_0^1(\Omega)\cap L^\infty(\Omega).
 \end{equation}
 Consider now $v\in \mathcal{V}(\Omega)$  which satisfies \eqref{5701}. Observe that
 $Z_\delta(u)v\in H_0^1(\Omega)\cap L^\infty(\Omega)$, and that
 \begin{equation}
\label{num25}
\begin{cases}
\dys-div\, {}^t\!A(x)D(Z_\delta(u)v)=\\
\dys=Z_\delta(u)(-div \, {}^t\!A(x)Dv)- \, {}^t\!A(x)DvDZ_\delta(u)\,+\\
+v(-div \, {}^t\!A(x)DZ_\delta(u))-\, {}^t\!A(x)DZ_\delta(u)Dv=\\
=\sum_{i\in I}Z_\delta(u)\hat{\varphi}_i (-div \,\hat{g}_i)+Z_\delta(u)\hat{f}-\, {}^t\!A(x)DvDZ_\delta(u)\,+\\
+v(-div\, {}^t\!A(x)DZ_\delta(u))-\, {}^t\!A(x)DZ_\delta(u)Dv  \mbox{ in } \mathcal{D}'(\Omega).
\end{cases}
\end{equation}
This proves that  $Z_\delta(u)v\in\mathcal V(\Omega)$, namely \eqref{5705}.\\
\smallskip

Finally, since $u\in H^1_{\mbox{\tiny loc}}(\Omega)$ by (\ref{sol1} {\it i}), one has 
$$R_\delta(u)\in H^1_{\mbox{\tiny loc}}(\Omega)\cap L^\infty(\Omega)$$ with 
$$
DR_\delta(u) =R'_\delta(u)Du=\frac{u}{\delta}\chi_{_{\{\delta< u< 2\delta\}}}Du  \mbox{ in } \mathcal{D}'(\Omega).
$$
Therefore (\ref{sol1} {\it iii}) with $k=\delta$ implies that $DR_\delta(u)\in (L^2(\Omega))^N$, which proves that $R_\delta(u)\in H^1(\Omega)\cap L^\infty(\Omega)$. 
Then using  the inequality $ 0\leq R_\delta(u)\leq G_{\frac{\delta}{2}}(u)$, property  (\ref{sol1} {\it iii}) with $\dys k=\frac \delta2$ and Lemma~\ref{lem0} of Appendix~\ref{appendixa} below completes the proof of \eqref{5706}.\\
\smallskip

\noindent{{{\bf Second step}}}. In this step we fix $\delta>0$ and $k$ such that
\begin{equation}
\label{kd}
k\geq 2\delta.
\end{equation}

Since $Z_\delta(u)v\in\mathcal{V}(\Omega)$ for every $v\in\mathcal{V}(\Omega)$ (see \eqref{5705}) with $Z_\delta(u)v\geq 0$ when $v\geq 0$, we can take $Z_\delta(u)v$ as test function in (\ref{sol2} {\it ii}), obtaining in view of \eqref{num25}
\begin{align}
\label{5711}
\begin{cases}
\dys \into  {}^t\!A(x) D(Z_\delta(u)v)DG_k(u)\,+\sum_{i\in I} \!\into \hat{g_i}D( Z_\delta(u)\hat{\varphi}_iT_k(u))\,+\\
\dys +\!\!\into Z_\delta(u)\hat{f}\,T_k(u)-\!\!\! \into \, {}^t\!A(x)Dv DZ_\delta(u)\,T_k(u)\,+\\
\dys+\into \, {}^t\!A(x)DZ_\delta(u)D(vT_k(u))-\dys\into \, {}^t\!A(x)DZ_\delta(u)  Dv\,T_k(u)\,=\\
\dys =\langle-div \, {}^t\!A(x)D(Z_\delta(u)v) ,G_k(u)\rangle_{H^{-1}(\Omega),H_0^1(\Omega)}\,+\\
\dys+\langle\langle -div\, {}^t\!A(x)D(Z_\delta(u)v) ,T_k(u) \rangle\rangle_\Omega=\\
\dys=\into F(x,u)Z_\delta(u)v.
\end{cases}
\end{align}

As far as the first term of the left-hand side of \eqref{5711} is concerned, we have, since $k\geq 2\delta$, 
\begin{equation}
\label{613bis}
\into \, {}^t\!A(x)D(Z_\delta(u)v)  DG_k(u)=0;
\end{equation}
indeed since $u\in H^1_{\mbox{\tiny loc}}(\Omega)$ by (\ref{sol1} {\it i}) one has 
\begin{align*}
\begin{cases}
\dys {}^t\!A(x)D(Z_\delta(u)v)  DG_k(u)=\\
\dys= - {}^t\!A(x) Du Du\, \frac{1}{\delta} \chi_{_{\{\delta< u< 2\delta\}}}\chi_{_{\{u>k\}}}+ {}^t\!A(x) Dv Du \, Z_\delta(u)\chi_{_{\{u>k\}}}\mbox{ in } \mathcal{D}'(\Omega),
\end{cases}
\end{align*}
where each term is zero almost everywhere when $k\geq 2\delta$.

As far as the fourth term of the left-hand side of \eqref{5711} is concerned, we have, since $u\in H^1_{\mbox{\tiny loc}}(\Omega)$ and since $k\geq 2\delta$, using \eqref{5703}, 
$$
DZ_\delta(u)\, T_k(u)=Du \,Z_\delta'(u)\, T_k(u)=-Du\, R_\delta'(u)=-DR_\delta(u) \mbox{ in } \mathcal{D}'(\Omega),
$$
which using the fact that $R_\delta(u)\in H_0^1(\Omega)$ (see \eqref{5706}) and \eqref{5701} implies that 
\begin{align}
\label{5713}
\begin{cases}
\dys- \into  {}^t\!A(x) Dv DZ_\delta(u)\, T_k(u)=  \into  {}^t\!A(x) Dv DR_\delta (u)=\\
\dys= \langle-div \, {}^t\!A(x)Dv,R_\delta(u)\rangle_{H^{-1}(\Omega),H_0^1(\Omega)}=\\
\dys= \sum_{i\in I} \into \hat{g_i}D(\hat{\varphi_i}R_\delta(u))+\into \hat{f}R_\delta(u).
\end{cases}
\end{align}

As far as the fifth term of the left-hand side of \eqref{5711} is concerned, we have, since $u\in H^1_{\mbox{\tiny loc}}(\Omega)$, 
\begin{align*}
\begin{cases}
\, {}^t\!A(x)DZ_\delta(u)D(vT_k(u))=\\
= {}^t\!A(x)DZ_\delta(u)Dv\, T_k(u)+ {}^t\!A(x)DZ_\delta(u)DT_k(u)\,v  \mbox{ in } \mathcal{D}'(\Omega),
\end{cases}
\end{align*}
which, since $k\geq 2\delta$, implies that
\begin{align}
\label{5712}
\begin{cases}
\dys\into \, {}^t\!A(x)DZ_\delta(u)D(vT_k(u))=\\
\dys=\into  {}^t\!A(x)DZ_\delta(u)Dv\,T_k(u)- \frac{1}{\delta}\int_{\{\delta<u<2\delta\}}  {}^t\!A(x)Du Du \, v. 
\end{cases}
\end{align}

Collecting together  \eqref{5711}, \eqref{613bis}, \eqref{5713} and \eqref{5712}, we have proved that
\begin{align}
\label{5714}
\begin{cases}
\dys\sum_{i\in I} \!\into \hat{g_i}D(\hat{\varphi_i}\big(Z_\delta(u)T_k(u)+R_\delta(u)\big))+\!\!\into \hat{f}\,\big(Z_\delta(u)T_k(u)+R_\delta(u)\big)=\\
\dys=\into F(x,u)  Z_\delta(u)v+\frac{1}{\delta}\int_{\{\delta<u<2\delta\}}  {}^t\!A(x)DuDu\,v .
\end{cases}
\end{align}

Since when $k\geq 2\delta$ one has (see \eqref{5703})
\begin{align*}
\begin{cases}
\dys   {}^t\!A(x) Du Du\, v \geq 0,\,\, \dys 0\leq Y_\delta(u)\leq \frac 32 \delta,\\
\dys Z_\delta(u)T_k(u)+R_\delta(u)=Z_\delta(u)u+R_\delta(u)=Y_\delta(u),\\
\dys \hat{g_i}D(\hat{\varphi_i}Y_\delta(u))=\hat{g_i}D\hat{\varphi_i}\,Y_\delta(u)+\hat{g_i}Du\,Z_\delta(u)\hat{\varphi_i},\\
\end{cases}
\end{align*}
estimate \eqref{5701bis}  follows from \eqref{5714}.

Proposition~\ref{prop3} is proved.
\end{proof}

\smallskip
 \noindent{} {\bf 5.4. A priori estimate of $\beta(u)$ in $H_0^1(\Omega)$ }$\mbox{ }$
 \label{sub54}
 
 This fourth a priori estimate is actually in some sense a regularity result, since it asserts that for every $u$ solution to problem \eqref{eqprima} in the sense of Definition~\ref{sol}, a certain function of $u$ actually belongs to $H_0^1(\Omega)$. This property will be used in the proofs of the Comparison Principle~\ref{prop0} and of the Uniqueness Theorem~\ref{uniqueness}.
 
 Define the function $\beta:s\in [0,+\infty[\rightarrow\beta(s)\in [0,+\infty[$ by 
\begin{equation}
\label{62bis}
\beta(s)=\int_0^s \sqrt{\Gamma'(t)}dt.
\end{equation}

 \begin{proposition}
\label{lem2}
Assume that the matrix $A$ and the function $F$ satisfy \eqref{eq0.0}, \eqref{car} and \eqref{eq0.1}. Then for every $u$ solution to problem \eqref{eqprima} in the sense of Definition~\ref{sol} one has
\begin{equation}
\label{510bis1}
\beta(u)\in H_0^1(\Omega),
\end{equation}
with
\begin{equation}
\label{num11bis}
\alpha\|D\beta(u)\|_{(L^2(\Omega))^N}^2\leq \|h\|_{L^1(\Omega)}.
\end{equation}
\end{proposition}
\begin{remark}{\bf (Formal proof of Proposition~\ref{lem2}).} 
\label{56bis}
Estimate \eqref{num11bis} can be obtained formally by taking $\Gamma(u)$ as test function in equation \eqref{eqprima}, using the coercivity \eqref{eq0.0} and the growth condition~(\ref{eq0.1}~{\it iii}), which implies that 
$$
0\leq F(x,u)\Gamma(u) \leq \frac{h(x)}{\Gamma(u)}\Gamma(u)\leq h(x). 
$$

This formal computation will be made mathematically correct in the proof below.
\qed
\end{remark}
\begin{proof}[{\bf Proof of  Proposition~\ref{lem2}}]
$\mbox{}$

As in \eqref{defs}, we define, for every $\delta$ and $k$ with $0<\delta<k$, the function $S_{\delta,k}$ as  
\begin{equation}
\label{562bis}
S_{\delta,k}(s)=\begin{cases}
 0 & \mbox{if } 0\leq s\leq \delta, \\
 s-\delta & \mbox{if }  \delta\leq s\leq k,\\
k-\delta & \mbox{if } k\leq s.
\end{cases}
\end{equation} 

As in the beginning of the proof of Proposition~\ref{prop1}, one can prove that the function $\Gamma(S_{\delta,k}(u))$ belongs to $\mathcal{V}(\Omega)$ and that (compare with \eqref{num7})
\begin{equation*}
 -div\, {}^t\!A(x)D \Gamma(S_{\delta,k}(u))= \psi_\delta(u) (-div\, {}^t\!A(x)D \Gamma(S_{\delta,k}(u))\quad \mbox{in } \mathcal{D}'(\Omega),
\end{equation*}
where $\psi_\delta(s)$ is a $C^1$ nondecreasing function such that 
$$
\psi_\delta(s)=0 \mbox{ for } 0\leq s\leq \frac{\delta}{2} \mbox{ and } \psi(s)=1 \mbox{ for } s\geq \delta.
$$

Since $\Gamma(S_{\delta,k}(u))\in\mathcal{V}(\Omega)$ and since $\Gamma(S_{\delta,k}(u))\geq 0$, we can use $\Gamma(S_{\delta,k}(u))$ as test function in  (\ref{sol2} {\it ii}). We get
\begin{align*}
\begin{cases}
\dys\into  {}^t\!A(x) D \Gamma(S_{\delta,k}(u)) D G_k(u)+\into  {}^t\!A(x) D \Gamma(S_{\delta,k}(u))D(\psi_\delta(u)T_k(u))=\\
\dys =\langle-div \, {}^t\!A(x)D\Gamma(S_{\delta,k}(u)),G_k(u)\rangle_{H^{-1}(\Omega),H_0^1(\Omega)}+\\\dys+\langle\langle-div \, {}^t\!A(x)D\Gamma(S_{\delta,k}(u)),T_k(u)\rangle\rangle_{\Omega}=\\
\dys=\into F(x,u) \Gamma(S_{\delta,k}(u)),
\end{cases}
\end{align*}
which can be written as 
\begin{align}
\label{sev4}
\begin{cases}
\dys\into  {}^t\!A(x) D \Gamma(S_{\delta,k}(u))D G_k(u)+\into  {}^t\!A(x) D \Gamma(S_{\delta,k}(u)) D\psi_\delta(u)\,T_k(u)\,+\\
\dys +\into  {}^t\!A(x) D \Gamma(S_{\delta,k}(u)) DT_k(u)\, \psi_\delta(u)=\into F(x,u) \Gamma(S_{\delta,k}(u)).
\end{cases}
\end{align}

Note  that the two first integrals are zero, since $D\Gamma(S_{\delta,k}(u))$ is zero outside of the set $\{\delta<u<k\}$, while $DG_k(u)$ is zero outside of the set $\{u>k\}$ and $D\psi_\delta(u)$ is zero outside of the set $\{u<\delta\}$. 

Note also that $\psi_\delta(u)=1$ on the set $\{u\geq \delta\}$ while $DS_{\delta,k}(u)=0$ outside of this set. Therefore the third term of \eqref{sev4} can be written as
\begin{align}
\label{sev5}
\begin{cases}
\dys \into  {}^t\!A(x) D \Gamma(S_{\delta,k}(u)) DT_k(u)\, \psi_\delta(u)=\\\vspace{0.1cm}
\dys = \into  {}^t\!A(x) D S_{\delta,k}(u) DT_k(u)\Gamma'(S_{\delta,k}(u))=\\
\dys \dys = \into  {}^t\!A(x) D S_{\delta,k}(u) D S_{\delta,k}(u) \Gamma'(S_{\delta,k}(u))\geq\\
\dys \geq  \alpha\into   |D S_{\delta,k}(u)|^2\Gamma'(S_{\delta,k}(u))=\alpha\into |D\beta(S_{\delta,k}(u))|^2. 
\end{cases}
\end{align}

For what concerns the right-hand side of \eqref{sev4}, we have, using the growth condition (\ref{eq0.1} {\it iii}) and the fact that $\Gamma$ is increasing,
\begin{align}
\label{sev6}
\dys F(x,u) \Gamma(S_{\delta,k}(u))\leq  \frac{h(x)}{\Gamma(u)}\Gamma(S_{\delta,k}(u))\leq h(x).
\end{align}

By \eqref{sev4}, \eqref{sev5} and \eqref{sev6} we get 
\begin{equation}
\label{566}
\alpha\into |D\beta(S_{\delta,k}(u))|^2 \leq \|h\|_{L^1(\Omega)} \,\, \forall \delta,\,\, 0<\delta<k.
\end{equation}

Passing to the limit as $\delta$ tends to zero gives implies that
\begin{equation}
\label{560bis}
\alpha\|D\beta(T_k(u))\|_{(L^2(\Omega))^N}^2\leq \|h\|_{L^1(\Omega)}\,\, \forall k>0.
\end{equation}

On the other hand, the fact that for every $\delta$ with $0<\delta<k$
\begin{align}
\label{564bis}
\begin{cases}
\dys |D\beta(S_{\delta,k}(u))|=|\beta'(S_{\delta,k}(u))\chi_{_{\{\delta<u<k\}}} Du|\leq \sup_{0\leq s\leq k}\beta'(s)\,|D G_\delta(u)|,\\
\dys 0\leq \beta(S_{\delta,k}(u))\leq \beta(G_{\delta}(u)),
\end{cases}
\end{align}
condition (\ref{sol1} {\it iii}) and Lemma~\ref{lem0} of Appendix~\ref{appendixa} below imply that $\beta(S_{\delta,k}(u))\in H_0^1(\Omega)\cap L^\infty(\Omega)$. The fact that $\beta(S_{\delta,k}(u))$ is bounded in $H_0^1(\Omega)$ by \eqref{566} and converges to $\beta(T_k(u))$ almost everywhere as $\delta$ tends to zero
then implies that 
\begin{equation}
\label{559bis}
\beta (T_k(u))\in H_0^1(\Omega).
\end{equation}

Since $\beta (T_k(u))\in H_0^1(\Omega)$ is bounded in $H_0^1(\Omega)$ in view of \eqref{560bis} and converges to $\beta(u)$ almost everywhere in $\Omega$ (note that $u$ is finite almost everywhere since $u\in L^2(\Omega)$), we have proved that $\beta(u)$ belongs to $H_0^1(\Omega)$ and satisfies \eqref{num11bis}.

Proposition~\ref{lem2} is proved.
\end{proof}

 \smallskip

\section{Proofs of the Stability Theorem~\ref{est} \\ and of the Existence  Theorem~\ref{EUS}}
\label{proofexistence}

\noindent{ } {\bf 6.1. Proof of the Stability Theorem~\ref{est}}
\label{proofstability}
$\mbox{}$\\
\smallskip

\noindent{{{\bf First step}}}.
Since for every $n$ the function $F_n(x,s)$ satisfies assumptions \eqref{car} and \eqref{eq0.1} for the same $h$ and the same $\Gamma$, every solution $u_n$ to problem \eqref{eqprima}$_n$ in the sense of Definition~\ref{sol} satisfies the a priori estimates \eqref{num3}, \eqref{53ter},  \eqref{619bis}, \eqref{516bis}, \eqref{111}, \eqref{512bis1} and \eqref{5701bis}.

Therefore $u_n$ is bounded in $L^2(\Omega)$ and in $H^1_{\mbox{\tiny loc}}(\Omega)$,  and there exist a subsequence, still labelled by $n$, and a function $u_\infty$ such that 
\begin{align}
\label{num2weak}
\begin{cases}
\dys u_n\rightharpoonup u_\infty \mbox{ in } L^2(\Omega)\mbox{  weakly, in } H^1_{\mbox{\tiny loc}}(\Omega) \mbox{ weakly and a.e. in }\Omega,\\
\dys  G_k(u_n)\rightharpoonup G_k(u_\infty) \, \mbox{ in } H_0^1(\Omega)\, \mbox{weakly}\,\, \forall k>0,\\
\dys \varphi T_k (u_n)\rightharpoonup \varphi T_k (u_{\infty}) \, \mbox{ in } H^1_0(\Omega)\, \mbox{weakly}\,\, \forall k>0,\,\,\forall \varphi\in H_0^1(\Omega)\cap L^\infty(\Omega).\\
\end{cases}
\end{align} Since $u_n\geq 0$, this function $u_\infty$ satisfies $u_\infty\geq 0$.

Therefore $u_\infty$ satisfies \eqref{sol1}.  

Fix now a function $v$ 
$$
v \in \mathcal V(\Omega),\,v\geq 0,\\ \dys\mbox{with }-div \, {}^t\!A(x)Dv=\sum_{i \in I} \hat{ \varphi}_i (-div\, \hat{ g}_i)+\hat{ f} \mbox{ in } \mathcal{D}'(\Omega).
$$

Using $v$ as test function in (\ref{sol2} {\it ii})$_n$, we obtain
\begin{align}
\label{72}
\begin{cases}
 \dys \into\, {}^t\!A(x)Dv DG_k(u_n)+\sum_{i\in I} \into \hat{g_i} D(\hat{\varphi_i} T_k(u_n))+\into \hat{f} T_k(u_n)=\\
 =\langle -div\, {}^t\!A(x)Dv, G_k(u_n) \rangle_{H^{-1}(\Omega),H_0^1(\Omega)}+\langle\langle -div\, {}^t\!A(x)Dv, T_k(u_n) \rangle\rangle_\Omega=\\
=\displaystyle \into F_n(x,u_n) v.
\end{cases}
\end{align}

Since the left-hand side of \eqref{72} is bounded  independently of $n$ for every $k>0$ fixed in view of the estimates \eqref{num3} and \eqref{516bis}, we have
$$\dys\into F_n(x,u_n)v\leq C(v)<+\infty \quad \forall n,$$ which using the almost everywhere convergence of $u_n$ to $u_\infty$, assumption \eqref{num1} on the functions $F_n$ and Fatou's Lemma gives 
 \begin{equation}
 \label{71bis}
\into F_\infty(x,u_\infty)v\leq C(v)<+\infty\quad \forall n,
\end{equation}
namely (\ref{sol2} {\it i})$_\infty$.

It remains to prove that (\ref{sol2} {\it ii})$_\infty$ holds true and that the convergences in \eqref{num2weak} are strong.\\
\smallskip 

\noindent{{{\bf Second step}}}.
For $\delta>0$ fixed, we write \eqref{72} as 
\begin{align}
\label{73}
\begin{cases}
\dys \into\, {}^t\!A(x)Dv DG_k(u_n)+  \displaystyle\sum_{i\in I} \into \hat{g_i} D(\hat{\varphi_i} T_k(u_n))+\into \hat{f} T_k(u_n)=\\
 =\langle -div\, {}^t\!A(x)Dv, G_k(u_n) \rangle_{H^{-1}(\Omega),H_0^1(\Omega)}+\langle\langle -div\, {}^t\!A(x)Dv, T_k(u_n) \rangle\rangle_\Omega=\\
=\displaystyle \int_{\Omega} F_n(x,u_n) Z_{\delta}(u_n)v+\int_{\Omega} F_n(x,u_n) (1-Z_{\delta}(u_n))v.
\end{cases}
\end{align}

It is easy to pass to the limit in the left-hand side of \eqref{73}, obtaining
\begin{align}
\label{74}
\begin{cases}
\dys \into\, {}^t\!A(x)Dv DG_k(u_n)+\displaystyle\sum_{i\in I} \into \hat{g_i} D(\hat{\varphi_i} T_k(u_n))+
\into \hat{f}T_k(u_n)\to \\
\to \dys \into\, {}^t\!A(x)Dv DG_k(u_\infty)+\sum_{i\in I} \into \hat{g_i} D(\hat{\varphi_i} T_k(u_\infty))+\into \hat{f}T_k(u_\infty)\\
\dys \mbox{as } n\to +\infty.
\end{cases}
\end{align}

Concerning the first term of the right-hand side of \eqref{73} we use the a priori estimate \eqref{5701bis},  namely
\begin{align*}
\begin{cases}
\dys\forall \delta>0,\,\,\int_{\{u_n\leq\delta\}} F_n(x,u_n)Z_{\delta}(u_n)v\leq\\
\dys\leq \frac 32 \left(\into \left|\sum_{i\in I} \hat{g_i} D\hat{\varphi_i}+ \hat{f}\right| \right)\delta+\sum_{i\in I} \into Z_\delta(u_n) \hat{g_i}Du_n \,\hat{\varphi_i},
\end{cases}
\end{align*}
in which we pass to the limit in $n$ for $\delta>0$ fixed. Since $Z_\delta(u_n) \hat{g_i}$ tends strongly to $Z_\delta(u_\infty) \hat{g_i}$ in $(L^2(\Omega))^N$ while $Du_n \,\hat{\varphi_i}$ tends weakly to $Du_\infty \,\hat{\varphi_i}$ in $(L^2(\Omega))^N$ (see \eqref{111}), we obtain
\begin{align}
\label{75}
\begin{cases}
\dys\forall\delta>0,\,\,\limsup_{n} \int_{\{u_n\leq\delta\}} F_n(x,u_n)Z_{\delta}(u_n)v\leq\\
\dys\leq \frac 32\left(\into \left|\sum_{i\in I} \hat{g_i} D\hat{\varphi_i}+ \hat{f}\right| \right)\delta+\sum_{i\in I} \into Z_\delta(u_\infty) \hat{g_i}Du_\infty \,\hat{\varphi_i}.
\end{cases}
\end{align}
 
  Since $Z_\delta(u_\infty)$ tends to $\chi_{_{\{u_\infty=0\}}}$ almost everywhere in $\Omega$  as $\delta$ tends to zero, and since $u_\infty\in H^1_{\mbox{\tiny loc}}(\Omega)$ implies that $D u_\infty=0$ almost everywhere on $\{x\in\Omega:u_\infty(x)=0\}$, the right-hand side of \eqref{75} tends to $0$ as $\delta$ tends to zero.

 We have proved that the first term of the right-hand side of \eqref{73} satisfies
\begin{equation}
\label{76}
\limsup_{n} \int_{\{u_n\leq\delta\}} F_n(x,u_n)Z_{\delta}(u_n)v\rightarrow0 \,\mbox{  as  }\, \delta\to0.
\end{equation}
\\
\smallskip

\noindent{{{\bf Third step}}}.
Let us now prove that
\begin{equation}
\label{88}
\int_{\{u_\infty=0\}} F_{\infty}(x,u_{\infty})v=0.
\end{equation}

In view of assumption \eqref{num1} on the convergence of the functions $F_n(x,s)$, of the continuity of the function $Z_\delta$, and of the almost everywhere convergence of $u_n$ to $u_\infty$, one has for every $\delta>0$ fixed 
\begin{equation}
\label{66665}
F_n(x,u_n)Z_{\delta}(u_n) v\rightarrow F_{\infty}(x,u_\infty)Z_{\delta}(u_\infty)v \mbox{ a.e. } x\in\Omega \mbox{ as } n\rightarrow +\infty.
\end{equation}
By Fatou's Lemma this implies that
\begin{equation}
\label{66666}
\into F_\infty(x,u_\infty)Z_{\delta}(u_\infty) v\leq \liminf_{n\to +\infty} F_n(x,u_n)Z_{\delta}(u_n)v \,\, \forall \delta>0.
\end{equation}

Since $Z_\delta(s)$ tends to $\chi_{_{\{s=0\}}}$ for every $s\geq 0$ as $\delta$ tends to zero and since $F_\infty(x,u_\infty)v\in L^1(\Omega)$ by \eqref{71bis}, the left-hand side of \eqref{66666} tends to $\dys \into F_\infty(x,u_\infty)\chi_{_{\{u_\infty=0\}}}v$ as $\delta$ tends to zero.

Combining the latest result with \eqref{66666} and \eqref{76} proves \eqref{88}.
\\
\smallskip

\noindent{{{\bf Fourth step}}}. 
 We finally pass to the limit in $n$ for $\delta>0$ fixed in the second term of the right-hand side of \eqref{73}, namely in the term $\dys \into F_n(x,u_n)(1-Z_\delta(u_n))v$.

For that we observe that 
$$
0\leq 1-Z_\delta(s)\leq \chi_{_{\{s\geq \delta\}}} \,\, \forall s\geq 0 \,\, \forall \delta>0,
$$
which combined with the growth condition (\ref{eq0.1} {\it iii}) and the fact that the function $\Gamma$ is increasing gives 
\begin{equation}
\label{61111}
0\leq F_n(x,u_n) (1-Z_\delta(u_n))v\leq \ \frac{h(x)}{\Gamma(u_n)}\chi_{_{\{u_n\geq\delta\}}}v\leq \frac{h(x)}{\Gamma(\delta)}v.
\end{equation}

Since the right-hand side of \eqref{61111} belongs to $L^1(\Omega)$, and since, for every $\delta>0$ fixed, one has, as in \eqref{66665},
$$
F_n(x,u_n)(1-Z_{\delta}(u_n)) v\rightarrow F_{\infty}(x,u_\infty)(1-Z_{\delta}(u_\infty))v \mbox{ a.e. } x\in\Omega \mbox{ as } n\rightarrow +\infty,
$$
Lebesgue's dominated convergence theorem implies that for every $\delta>0$ fixed one has 
\begin{equation}
\label{61112}
\into F_n(x,u_n)(1-Z_{\delta}(u_n)) v\rightarrow \into F_{\infty}(x,u_\infty)(1-Z_{\delta}(u_\infty))v \mbox{ as } n\rightarrow +\infty.
\end{equation}

Since the right-hand side of \eqref{61112} tends to $\dys\into F_{\infty}(x,u_\infty)(1-\chi_{_{\{u_\infty=0\}}})v $ as $\delta$ tends to zero, which is nothing but $\dys\into F_\infty(x,u_\infty)v$ in view of \eqref{88}, we have proved that 
\begin{equation}
\label{61113}
\lim_n \into F_n(x,u_n)(1-Z_\delta(u_n))v \rightarrow \into F_{\infty}(x,u_{\infty})v \,\mbox{ as }\, \delta\to0.
\end{equation}
\\
\smallskip

\noindent{{{\bf Fifth step}}}.
Collecting the results obtained in \eqref{74}, \eqref{76} and \eqref{61113}, we pass to the limit in each term of \eqref{73}, first in $n$ for $\delta>0$ fixed, and then in $\delta$. This proves that for every $v\in \mathcal{V}(\Omega), v\geq0$, one has
\begin{align*}
\begin{cases}
\dys \into\, {}^t\!A(x)Dv DG_k(u_\infty)+\sum_{i\in I} \into \hat{g_i} D(\hat{\varphi_i} T_k(u_\infty))+\into \hat{f}T_k(u_\infty)=\\
\dys =\into F_\infty(x,u_\infty)v,
\end{cases}
\end{align*}
which is nothing but (\ref{sol2} {\it ii})$_\infty$.

We have proved that $u_\infty$ is a solution to problem \eqref{eqprima}$_\infty$ in the sense of Definition~\ref{sol}.

It only remains to prove that the convergences in \eqref{num2weak} are strong (see \eqref{num2}). To this aim it is sufficient to prove the two following strong convergences
\begin{equation}\label {gikappa}
\dys  G_k(u_n)\rightarrow  G_k(u_\infty) \, \mbox{ in } H_0^1(\Omega)\, \mbox{strongly }\,\, \forall k>0,
\end{equation}
\begin{align}\label {fi}
\begin{cases}
\dys \varphi DT_k (u_n)\rightarrow \varphi DT_k (u_{\infty}) \, \mbox{ in } (L^2(\Omega))^N\, \mbox{strongly}\,\, \forall k>0,\,\\
\forall \varphi\in H_0^1(\Omega)\cap L^\infty(\Omega).
\end{cases}
\end{align}
Indeed one has $u_n=T_k(u_n)+G_k(u_n)$, and the strong convergence of $T_k(u_n)$ in $L^2(\Omega)$ then follows from Lebesgue's dominated convergence Theorem and from the almost convergence of $u_n$ (see \eqref{num2weak}). These facts in turn 
imply the strong convergences of $u_n$ in $L^2(\Omega)$, of $\varphi T_k(u_n)$ in $H_0^1(\Omega)$, and of $u_n$ in $H^1_{\mbox{\tiny loc}}(\Omega)$.
\\
\smallskip

\noindent{{{\bf Sixth step}}}. As far as \eqref{gikappa} is concerned, the strong convergence follows from the energy equality $\eqref{condfc3}_n$, namely the fact that 
\begin{equation}\label{54n}
\into A(x) DG_k(u_n) D G_k(u_n)=\into F_n(x,u_n)G_k(u_n) \,\,  \forall k>0.
\end{equation}

It is easy to pass to the limit in the right-hand side of \eqref{54n}. Indeed the inequality
$$
0\leq F_n(x, u_n) G_k(u_n)\leq  \frac{h(x)}{\Gamma(u_n)}\chi_{_{\{u_n\geq k\}}}G_k(u_n)\leq {{h(x)}\over{\Gamma(k)}} G_k(u_n)
$$
and the boundedness of $G_k(u_n)$ in $H^1_0(\Omega)$ (see \eqref{num2weak}) imply that the sequence $F_n(x, u_n) G_k(u_n)$ is equintegrable in $n$; since $$F_n(x, u_n) G_k(u_n)\rightarrow F_\infty(x, u_\infty) G_k(u_\infty)\,\,\mbox{ a.e. } x\in\Omega, $$ using Vitali's Theorem one obtains that for every $k>0$ fixed
\begin{equation}\label{rhs}
\dys \int_\Omega F_n(x, u_n) G_k(u_n) \rightarrow \dys \int_\Omega F_\infty(x, u_\infty) G_k(u_\infty) \mbox{ as } n\to +\infty. 
\end{equation}

On the other hand, the energy equality $\eqref{condfc3}_\infty$ asserts that
\begin{equation}\label{54infty}
\into A(x) DG_k(u_\infty) D G_k(u_\infty)=\into F_\infty(x,u_\infty)G_k(u_\infty) \,\,  \forall k>0.
\end{equation}

Collecting together \eqref{54n}, \eqref{rhs} and \eqref{54infty}, one has  for every $k>0$ fixed
$$
\into A(x) DG_k(u_n) D G_k(u_n)\rightarrow \into A(x) DG_k(u_\infty) D G_k(u_\infty)\,\,  \mbox{as}\,\, n\rightarrow +\infty ,
$$
which, with the weak convergence in $(L^2(\Omega))^N$ of $DG_k(u_n)$ to $DG_k(u_\infty)$ (see \eqref{num2weak}), implies the strong convergence \eqref{gikappa}.
\\
\smallskip

\noindent{{{\bf Seventh step}}}.  As far as  \eqref{fi} is concerned, writing equation  \eqref{num14}  for $F=F_n$ and $u=u_n$, one obtains
\begin{align}
\label{num18}
\begin{cases}
\dys \into  A(x) DT_k(u_n)DT_k(u_n)\, \varphi^2=\\
\dys =-2\into A(x) DT_k(u_n)D\varphi\,\varphi T_k(u_n)-2\into A(x) DG_k(u_n)D\varphi\,\varphi T_k(u_n)+\\
\dys+\into F_n(x,u_n)\varphi^2T_k(u_n)\quad \dys\forall \varphi\in H_0^1(\Omega)\cap L^\infty(\Omega).
\end{cases}
\end{align}

Let us pass to the limit in the right-hand side  of \eqref{num18} as $n$ tends to $+\infty$. In view of \eqref{num2weak} one has
\begin{align}
\label{622new}
\begin{cases}
\dys -2\into A(x) DT_k(u_n)D\varphi\,\varphi T_k(u_n)-2\into A(x) DG_k(u_n)D\varphi\,\varphi T_k(u_n)\rightarrow\\
\dys\rightarrow-2\into A(x) DT_k(u_\infty)D\varphi\,\varphi T_k(u_\infty)-2\into A(x) DG_k(u_\infty)D\varphi\,\varphi T_k(u_\infty)\\
\mbox{as } n\rightarrow +\infty.
\end{cases}
\end{align}

As far as the last integral of the right-hand side of \eqref{num18} is concerned, we write, for every fixed $\delta>0$, 
\begin{align}
\label{622bis}
\begin{cases}
\dys \into F_n(x,u_n)\varphi^2 T_k(u_n)=\\
\dys=\into F_n(x,u_n) Z_\delta(u_n) \varphi^2 T_k(u_n)+\into F_n(x,u_n)(1-Z_\delta(u_n))\varphi^2 T_k(u_n).
\end{cases}
\end{align}

For the first term of the right-hand side of \eqref{622bis}, we use the a priori estimate \eqref{5701bis} with $v=\varphi^2$; indeed $\varphi^2\in \mathcal{V}(\Omega)$ (see \eqref{condv5}) with 
$$
-div\, {}^t\!A(x)D\varphi^2=\hat\varphi (-div\, \hat{g})+\hat{f}\,\, {\rm in }\,\mathcal{D}'(\Omega),
$$
with $\hat\varphi=2\varphi$, $\hat{g}=\, {}^t\!A(x)D\varphi$ and $\hat{f}=-2\, {}^t\!A(x)D\varphi D\varphi$; this yields, since $\hat gD\hat\varphi +\hat f=0$,
\begin{align}
\label{622ter}
\begin{cases}\vspace{0.1cm}
\dys 0\leq\into F_n(x,u_n)Z_\delta(u_n)\varphi^2 T_k(u_n)\leq k\into F_n(x,u_n)Z_\delta(u_n)\varphi^2 \leq\\\vspace{0.1cm}
\dys \leq 2k \into Z_\delta(u_n) \,{}^t\!A(x)D\varphi Du_n\,\varphi=\\
\dys = 2k \into Z_\delta(u_n) \, {}^t\!A(x)D\varphi DT_k(u_n)\,\varphi+2k \into Z_\delta(u_n)\,  {}^t\!A(x)D\varphi DG_k(u_n)\,\varphi.
\end{cases}
\end{align}
We pass to the limit in the right-hand side of \eqref{622ter} first for $\delta>0$ fixed as $n$ tends to $+\infty$ thanks to \eqref{num2weak}, and then as $\delta$ tends to $0$. Since
$
Z_\delta(u_\infty)$ tends to $\chi_{_{\{u_\infty=0\}}}$ in $L^\infty(\Omega)$  weak-star as $\delta$ tends to zero, and since $Du_\infty=0$ a.e. on $\{u_\infty=0\}$,
we obtain that 
\begin{equation}
\label{622quatro}
\limsup_{n} \into F_n(x,u_n)Z_\delta(u_n) \varphi^2 T_k(u_n)\rightarrow 0 \mbox{ as } \delta\rightarrow 0,
\end{equation}

For the second term of the right-hand side of \eqref{622bis}, we repeat the proof that we did above to prove \eqref{61113}, and we obtain that
\begin{equation}
\label{622quinto}
\lim_{n} \into F_n(x,u_n)(1-Z_\delta(u_n))\varphi^2 T_k(u_n)\rightarrow \into F_\infty(x,u_\infty)\varphi^2T_k(u_\infty) \mbox{ as } \delta\rightarrow 0.
\end{equation}

On the other hand, writing equation \eqref{num14} for $F=F_\infty$ and $u=u_\infty$, one has
\begin{align}
\label{622sexto}
\begin{cases}
\dys \into  A(x) DT_k(u_\infty)DT_k(u_\infty)\, \varphi^2=\\
\dys =-2\into A(x) DT_k(u_\infty)D\varphi\,\varphi T_k(u_\infty)-2\into A(x) DG_k(u_\infty)D\varphi\,\varphi T_k(u_\infty)\,+\\
\dys+\into F_\infty(x,u_\infty)\varphi^2T_k(u_\infty)\quad \dys\forall \varphi\in H_0^1(\Omega)\cap L^\infty(\Omega).
\end{cases}
\end{align}

Collecting together \eqref{num18}, \eqref{622new}, \eqref{622bis}, \eqref{622quatro}, \eqref{622quinto} and \eqref{622sexto}, one has for every $k>0$ fixed
\begin{align*}
\begin{cases}
\dys \into  A(x) DT_k(u_n)DT_k(u_n)\, \varphi^2\rightarrow \into  A(x) DT_k(u_\infty)DT_k(u_\infty)\, \varphi^2\\
\mbox{as } n\rightarrow +\infty,
\end{cases}
\end{align*}
which, with the weak convergence in $(L^2(\Omega))^N$ of $\varphi DT_k(u_n)$ to $\varphi DT_k(u_\infty)$ implies the strong convergence \eqref{fi}.

This completes the proof of the Stability Theorem~\ref{est}.
\qed
\\
\smallskip

\noindent{ }  {\bf 6.2. Proof of  the Existence Theorem~\ref{EUS}}
\label{profiste}
$\mbox{  }$

Consider the problem  
 \begin{equation}
 \label{81}
\begin{cases}
u_n\in H_0^1(\Omega),\\
\displaystyle \into A(x) Du_n D\varphi=\into T_n(F(x,u_n^+))\varphi \quad \forall \varphi \in H_0^1(\Omega),
\end{cases} 
\end{equation}
where $T_n$ is the truncation at height $n$.

Since 
$$T_n(F(x,s^+)):(x,s)\in\Omega\times]-\infty,+\infty[\rightarrow T_n(F(x,s^+))\in [0,n]$$ is a Carath\'eodory function which is almost everywhere bounded by $n$, Schauder's fixed point theorem implies that problem \eqref{81} has at least a solution. Moreover, since $T_n(F(x,s^+))\geq 0$, this solution is nonnegative by the weak maximum principle, so that $u_n\geq 0$, and $T_n(F(x,u_n^+))=T_n(F(x,u_n))$.

Define now the function $F_n$ by
$$
F_n(x,s)=T_n(F(x,s)) \mbox{ a.e. } x\in\Omega,\,\, \forall s\in[0,+\infty[.
$$

For every given $n$, the function $F_n$ is bounded, and it is easy to see that every $u_n$ solution to \eqref{81} is a solution of problem \eqref{eqprima}$_n$ in the sense of Definition~\ref{sol}, where \eqref{eqprima}$_n$ is the problem \eqref{eqprima} where the function $F$ has been replaced by $F_n$: indeed $u_n\geq 0$ belongs to $H_0^1(\Omega)$, and therefore satisfies \eqref{sol1} and (\ref{sol2} {\it i}); $u_n$ also satisfies  (\ref{sol2} {\it ii}) in view of 
\begin{equation}\label{640}
\begin{cases}
\langle\langle -div\, {}^t\!A(x)Dv, T_k(u_n) \rangle\rangle_\Omega=\langle -div\, {}^t\!A(x)Dv, T_k(u_n)\rangle_{H^{-1}(\Omega),H^1_0(\Omega) } \\\, \forall v\in\mathcal{V}(\Omega), 
\end{cases}
\end{equation}
which results of \eqref{classic} by taking $y=T_k(u_n)\in H_0^1(\Omega)\cap L^\infty(\Omega)$.

It is clear that the functions $F_n$ satisfy \eqref{car} and \eqref{eq0.1} with the functions $h$ and $\Gamma$ which appear in the definition of the function $F$. Moreover it is not difficult to verify that the function $F_n$ satisfy \eqref{num1} with $F_\infty$ given by
$F_{\infty}(x,s)=F(x,s)$.

The Stability Theorem~\ref{est} then implies that there exists a subsequence  and a function $u_\infty$ which is a solution to problem \eqref{eqprima} in the sense of Definition~\ref{sol} such that the convergences \eqref{num2} hold true.

This proves the Existence Theorem~\ref{EUS}.
\qed
\smallskip

\section{ Comparison Principle\\ and proof of the Uniqueness Theorem~\ref{uniqueness}}
\label{comparison}
In this section we prove  a comparison result which uses assumption \eqref{eq0.2}, namely the fact that $F(x,s)$ is nonincreasing with respect to $s$.  Note that we never use this assumption in  the rest of the present paper except as far as the Uniqueness Theorem~\ref{uniqueness} is concerned.
\begin{proposition}{\bf (Comparison Principle).}
\label{prop0} 
Assume that the matrix $A$ satisfies \eqref{eq0.0}. Let $F_1(x,s)$ and $F_2(x,s)$ be two functions satisfying \eqref{car} and \eqref{eq0.1} (possibly for different functions $h$ and $\Gamma$). Assume moreover that 
\begin{equation}
\label{80}
\mbox{either } F_1(x,s) \mbox{ or } F_2(x,s) \mbox{ is nonincreasing, i.e. satisfies } \eqref{eq0.2},
\end{equation}
 and that  
\begin{equation}
                              \label{condfc}
F_1(x,s)\leq F_2(x,s)\,\,\, \mbox{a.e. } x\in\Omega,\quad \forall s\geq0.
\end{equation}

 Let $u_1$ and $u_2$ be any solutions in  the sense of Definition~\ref{sol} to problems \eqref{eqprima}$_1$ and \eqref{eqprima}$_2$, where \eqref{eqprima}$_1$ and \eqref{eqprima}$_2$ are \eqref{eqprima} with $F(x,s)$ replaced respectively by $F_1(x,s)$ and $F_2(x,s)$. 
 Then one has
\begin{equation}
                               \label{condfc2}
u_1(x)\leq u_2(x) \mbox{ a.e. } x\in\Omega.
\end{equation}
\end{proposition}
\begin{remark}
The proof of the Comparison Principle of Proposition~\ref{prop0} is based on the use of the test function $(B_1(T_k^+(u_1-u_2)))^2$. A similar test function has been used by L.~Boccardo and J.~Casado-D\'iaz in \cite{BC} to prove the uniqueness of solution to problem \eqref{eqprima} obtained by approximation.

Note that the use of this test function is allowed by the regularity property \eqref{510bis1} proved in Proposition~\ref{lem2}.
\qed
\end{remark}

\begin{proof}[{\bf Proof of the Uniqueness Theorem~\ref{uniqueness}}]$\mbox{  }$

Applying the Comparison Principle to the case where $F_1(x,s)=F_2(x,s)=F(x,s)$ with $F(x,s)$ satisfying \eqref{eq0.2} immediately proves the Uniqueness Theorem~\ref{uniqueness}.
\end{proof}

\begin{proof}[{\bf Proof of  Proposition~\ref{prop0}}]$\mbox{  }$

\smallskip

\noindent{{{\bf First step}}}.
Let $k>0$ be fixed. Define $\varphi$ by 
\begin{equation}
\label{899}
\varphi= B_1(T_k^+(u_1-u_2)),  
\end{equation}
where $B_1:s\in[0,+\infty[\rightarrow B_1(s)\in[0,+\infty[$ is the function defined by
$$
B_1(s)=\int_0^s \beta_1(t)dt\quad \forall s\geq 0,
$$
where $\dys\beta_1(s)=\int_0^s \sqrt{\Gamma'_1(t)}dt$ and $\Gamma_1$ is the function for which $F_1$ satisfies \eqref{eq0.1}.

In this step we will prove that
\begin{equation}
\label{8100}
\varphi\in H_0^1(\Omega)\cap L^\infty(\Omega).
\end{equation}

Since $u_1$ and $u_2$ belong to $H^1_{\mbox{\tiny loc}}(\Omega)$, one  has $T_k^+(u_1-u_2)\in H^1_{\mbox{\tiny loc}}(\Omega)\cap L^\infty(\Omega)$; since $\beta_1\in C^0([0,+\infty[)$, the function $\varphi$ belongs to $H^1_{\mbox{\tiny loc}}(\Omega)$ and one has
\begin{equation}
\label{8200}
D\varphi=\beta_1(T_k^+(u_1-u_2))\chi_{_{\{0<u_1-u_2<k\}}}(Du_1-Du_2)\,  \mbox{ in } \, \mathcal{D}'(\Omega).
\end{equation}
Since $0\leq T_k^+(s_1-s_2)\leq T_k(s_1)$ for $s_1\geq 0$ and $s_2\geq 0$ and since $\beta_1$ is nondecreasing, this implies that
$$
|D\varphi|\leq \beta_1(T_k(u_1))(|Du_1|+|Du_2|).
$$
But in view of \eqref{510bis1}, $\beta_1(T_k(u_1))$ belongs to $H_0^1(\Omega)\cap L^\infty(\Omega)$, and then the properties \eqref{3111} for $u_1$ and $u_2$ imply that  $D\varphi\in (L^2(\Omega))^N$, and therefore that $\varphi\in H^1(\Omega)\cap L^\infty(\Omega)$.

Since $B_1$ is nondecreasing one has $0\leq \varphi\leq B_1(T_k(u_1))$.  We now claim that
\begin{equation}
\label{8101-}
B_1(T_k(u_1))\in H_0^1(\Omega),
\end{equation}
which by Lemma~\ref{lem0} of Appendix~\ref{appendixa} below completes the proof of \eqref{8100}.

Let us now prove \eqref{8101-}. As in \eqref{564bis}, we observe that for every $\delta$ with $0<\delta<k$ and for the function $S_{\delta,k}$ defined by \eqref{562bis} one has
\begin{align*}
\begin{cases}
\dys |DB_1(S_{\delta,k}(u_1))|=|\beta_1(S_{\delta,k}(u_1))\chi_{_{\{\delta<u_1<k\}}} Du_1|\leq \sup_{0\leq s\leq k} \beta_1(s)\,|D G_\delta(u_1)|,\\
\dys 0\leq B_1(S_{\delta,k}(u_1))\leq B_1(G_{\delta}(u_1)).
\end{cases}
\end{align*}
Then Lemma~\ref{lem0} implies that $B_1(S_{\delta,k}(u_1))\in H_0^1(\Omega)\cap L^\infty(\Omega)$. But $|DB_1(S_{\delta,k}(u_1))|\leq \beta_1(T_k(u_1))\,|Du_1|$, which belongs to $L^2(\Omega)$ because of \eqref{510bis1} and \eqref{3111}. Therefore $\beta_1(S_{\delta,k}(u_1))$ converges to $\beta_1(T_k(u_1))$ in $H^1(\Omega)$ weakly as $\delta$ tends to zero.

This implies \eqref{8101-}.

\smallskip

\noindent{{{\bf Second step}}}.
Since $\varphi^2\in \mathcal{V}(\Omega)$ in view of \eqref{8100} and of Remark~\ref{examples}~{\it ii}), we can  take $\varphi^2$ as test function in (\ref{sol2} {\it ii})$_1$ and (\ref{sol2} {\it ii})$_2$. Taking the difference of these two equations, we get
\begin{align}
\label{casopo}
\begin{cases}
\dys2\into \varphi \, {}^t\!A(x) D\varphi D(G_k(u_1)-G_k(u_2))\,+\\\vspace{0.1cm}
+ \dys2\into  \, {}^t\!A(x) D\varphi D(\varphi(T_k(u_1)-T_k(u_2)))\,-\\
- \dys2\into  \, {}^t\!A(x) D\varphi D\varphi\,(T_k(u_1)-T_k(u_2))=\\
=\langle -div\, {}^t\!A(x)D\varphi^2, G_k(u_1)-G_k(u_2) \rangle_{H^{-1}(\Omega),H_0^1(\Omega)}\,+\\+\langle\langle -div\, {}^t\!A(x)D\varphi^2, T_k(u_1)-T_k(u_2) \rangle\rangle_\Omega=\\
\dys = \into (F_1(x,u_1)-F_2(x,u_2))\,\varphi^2.
\end{cases}
\end{align}

Expanding in $L^1_{\mbox{\tiny loc}}(\Omega)$ the integrands of the three first lines of \eqref{casopo}, one realizes that their sum is nothing but $2 \varphi \, {}^t\!A(x) D\varphi D(u_1-u_2)$, which belongs to $L^1(\Omega)$ in view of \eqref{3111}. Therefore one has
\begin{equation}
\label{8300}
2\into \varphi \, {}^t\!A(x) D\varphi D(u_1-u_2)= \into (F_1(x,u_1)-F_2(x,u_2))\,\varphi^2,
\end{equation}
which is formally easily obtained by taking $\varphi^2$ as test function in \eqref{eqprima}$_1$ and \eqref{eqprima}$_2$ and making the difference.

\smallskip

\noindent{{{\bf Third step}}}.
Let us now prove that for $\varphi$ given by \eqref{899} one has
\begin{equation}
\label{87}
(F_1(x,u_1)-F_2(x,u_2))\,\varphi^2\leq 0\, \mbox{ a.e. in } x\in\Omega.
\end{equation}

Since $F_1(x,u_1)$ and $F_2(x,u_2)$ belong to $L^1_{\mbox{\tiny loc}}(\Omega)$ (see \eqref{3107}) and are therefore finite almost everywhere, one has 
$$
(F_1(x,u_1)-F_2(x,u_2))\,\varphi^2=0 \mbox{  on the set where } \varphi^2=0,
$$
since there are no indeterminacies of the types $(\infty-\infty)$ and $\infty\times 0$ in the latest formula.

On the set where $\varphi^2>0$, one has $u_1>u_2$. If $F_1(x,s)$ is nonincreasing with respect to $s$, one has, using first this nonincreasing character and then assumption \eqref{condfc}, 
\begin{align*}
\begin{cases}
\dys F_1(x,u_1)-F_2(x,u_2)\leq F_1(x,u_2)-F_2(x,u_2)\leq 0\\
\dys \mbox{ on the set where } \varphi^2>0,
\end{cases}
\end{align*}
since no indeterminacy of the type $(\infty-\infty)$ appears.

This implies that 
$$
(F_1(x,u_1)-F_2(x,u_2))\varphi^2\leq 0 \mbox{ on the set where } \varphi^2>0.
$$

The case where $F_2(x,s)$ is nonincreasing with respect to $s$ is similar, using first assumption \eqref{condfc} and then this nonincreasing character.

We have proved \eqref{87}.

\smallskip

\noindent{{{\bf Fourh step}}}. Collecting together \eqref{8300} and \eqref{87}  gives 
$$
2\into B_1(T_k^+(u_1-u_2)) \beta_1(T_k^+(u_1-u_2) ) \, {}^t\!A(x) DT_k^+(u_1-u_2) D(u_1-u_2)\leq 0.
$$

Defining the function $M_1$ by 
$$
M_1(s)=\int_0^s \sqrt{B_1(t)\beta_1(t)}dt \quad \forall s>0,
$$
this implies in view of the coercivity \eqref{eq0.0} of the matrix $A$ that $D M_1(T_k^+(u_1-u_2))=0$ on $\Omega$. Therefore  $M_1(T_k^+(u_1-u_2))$ is a constant in each connected component $\omega$ of $\Omega$, and since the function $M_1$ is strictly increasing, there exists a nonnegative constant $C_\omega$ such that  $T_k^+(u_1-u_2)=C_\omega$ in~$\omega$, and therefore  $\beta_1(T_k^+(u_1-u_2))=\beta_1(C_\omega)$ in $\omega$. Since  $T_k^+(u_1-u_2)\leq T_k(u_1)$ in $\Omega$ and since the function $\beta_1$ is nondecreasing, one has 
$$
0\leq \beta_1(C_\omega)=\beta_1 (T_k^+(u_1-u_2))\leq \beta_1(T_k(u_1))\mbox{ in } \omega.
$$
By the regularity property \eqref{510bis1}, the function $\beta_1(T_k(u_1))$ belongs to $H_0^1(\Omega)$ and therefore to $H_0^1(\omega)$ for every connected component $\omega$ of~$\Omega$. By Lemma~\ref{lem0} this implies that $\beta_1(C_\omega)\in H_0^1(\Omega)$, and then that $C_\omega=0$. This proves that $T_k^+(u_1-u_2)=0$ in $\Omega$, which implies  \eqref{condfc2} since $k > 0$.
\end{proof}
\smallskip

\appendix
\setcounter{equation}{0}

\section{An useful lemma}
\label{appendixa}

In this Appendix we state and prove the following lemma which is used many times in the present paper.


\begin{lemma}
\label{lem0}
If $y\in H^1(\Omega)$ and if there exists $\underline{y}$ and $\overline{y}\in H_0^1(\Omega)$ such that
$$\underline{y}\leq y\leq \overline{y}\quad \mbox{a.e. in }\Omega,$$    then $y\in H^1_0(\Omega)$.
\end{lemma}
\begin{remark}
Lemma~\ref{lem0} is straightforward when $\partial\Omega$ is sufficiently smooth so that the traces of the functions of $H^1(\Omega)$ are defined. Note that we did not assume any smoothness of $\partial\Omega$.
\qed
\end{remark}
\begin{proof}[{\bf Proof of Lemma~\ref{lem0}}]
$\mbox{}$

Since $0\leq y-\underline{y}\leq \overline{y}-\underline{y}$, it is sufficient to consider the case where $\underline{y}=0$, namely the case where 
$$
y\in H^1(\Omega)\, \mbox{ with } 0\leq y\leq\overline{y}, \mbox{ where } \overline{y}\in H_0^1(\Omega).
$$ 

Since $\overline{y}\in H_0^1(\Omega)$, there exists a sequence  $\phi_n\in \mathcal{D}(\Omega)$, such that $$\phi_n\rightarrow \overline{y}\quad \mbox{in } H_0^1(\Omega).$$ 


The function $y_n$ defined by
$$
y_n=\phi_n^+-(\phi_n^+-y)^+$$belongs to $H^1(\Omega)$ and has compact support (included in the support of $\phi_n$). Therefore 
$$
y_n\in H_0^1(\Omega),
$$
and 
$$
 y_n\rightarrow \overline{y}^+-(\overline{y}^+-y)^+= y\quad \mbox{in }  H^1(\Omega)\quad  \mbox{as } n\rightarrow +\infty.
$$
This implies that $y\in H_0^1(\Omega)$.
\end{proof}

\smallskip

\bigskip

 \noindent \textbf{Acknowledgments}.\label{acknowledgments}   The authors would like to thank Gianni Dal~Maso and Luc Tartar for their friendly help, and to thank Lucio Boccardo, Juan Casado-D\'iaz and Luigi Orsina for having introduced them  to the topics of the present work. They also would like to thank their own institutions (Dipartimento di Scienze di Base e Applicate per l'Ingegneria, Facolt\`a di Ingegneria Civile e Industriale, Sapienza Universit\`a di Roma, Departamento de Matem\'atica Aplicada y Estad\'istica, Universidad Polit\'ecnica de Cartagena, and Laboratoire Jacques-Louis Lions, Universit\'e Pierre et Marie Curie Paris VI et CNRS) for providing the support of reciprocal visits which allowed them to perform the present work. The work of Pedro J. Mart\'inez-Aparicio has been partially supported by the grant MTM2015-68210-P of the Spanish Ministerio de Econom\'ia y Competitividad (MINECO-FEDER), the FQM-116 grant of the Junta de Andaluc\'ia and the grant Programa de Apoyo a la Investigaci\'on de la Fundaci\'on S\'eneca-Agencia de Ciencia y Tecnolog\'ia de la Regi\'on de Murcia 19461/PI/14.

\end{document}